\documentclass[a4paper,oneside,12pt]{amsart}

\usepackage{geometry}
\makeatletter
\def\@seccntformat#1{%
  \protect\textup{\protect\@secnumfont
    \ifnum\pdfstrcmp{subsection}{#1}=0 \bfseries\fi
    \csname the#1\endcsname
    \protect\@secnumpunct
  }%
}  
\makeatother

\title[Fractional Sobolev-Choquard critical systems]{Fractional Sobolev-Choquard critical systems with Hardy term and weighted singularities}

\author[Assun\c{c}\~{a}o]{Ronaldo B. Assun\c{c}\~{a}o}
\address{Ronaldo B. Assun\c{c}\~{a}o \hfill\break\indent
Departamento de Matem\'{a}tica\,---\,
Universidade Federal de Minas Gerais, 
UFMG}
\email{ronaldo@mat.ufmg.br}

\author[Miyagaki]{Ol\'{\i}mpio H. Miyagaki}
\thanks{Ol\'{\i}mpio H. Miyagaki was supported by Grant 2022/16407-1\,---\,S\~{a}o Paulo Research Foundation (FAPESP) and Grant 303256/2022-2\,---\,CNPq/Brazil.}
\address{Ol\'{\i}mpio H. Miyagaki \hfill\break\indent
Departamento de Matem\'{a}tica\,---\,
Universidade Federal de S\~{a}o Carlos, 
UFSCar} 
\email{ohmiyagaki@gmail.com}

\author[Siqueira]{Rafaella F. S. Siqueira}
\thanks{Rafaella F. S. Siqueira was partially supported by CNPq/Brazil.}
\address{Rafaella F. S. Siqueira 
\hfill\break\indent
Departamento de Matem\'{a}tica\,---\,
Centro Federal de Educa\c{c}\~{a}o Tecnol\'{o}gica, 
CEFET}
\email{rafaella.siqueira@cefetmg.br}

\date{Belo Horizonte, \today}



\usepackage[utf8]{inputenc}
\usepackage[T1]{fontenc}
\usepackage{comment}
\usepackage{amsmath}
\usepackage{amsthm}
\usepackage{amssymb}
\usepackage{amsfonts}

\usepackage{bm}

\usepackage{mathtools}
\usepackage{xfrac}
\usepackage{pdfpages}
\usepackage{fancyhdr}
\usepackage{physics}

\usepackage[inline]{enumitem}

\usepackage{color}
\usepackage{graphicx}
\usepackage[mathscr]{euscript}
\usepackage{bigints}

\usepackage{xspace}

\newcommand\ie{i.e.,\xspace}

\newtheorem{theorem}{Theorem}

\newtheorem{proposition}[theorem]{Proposition}
\newtheorem{corollary}[theorem]{Corollary}
\newtheorem{lemma}[theorem]{Lemma}

\theoremstyle{definition}

\theoremstyle{remark}

\newtheorem{claim}{Claim}

\makeatletter

\newcommand*{\sen}{%
  \mathop{\kern\z@\operator@font sen}\!{}%
}
\makeatother

\usepackage{hyperref}
\hypersetup{%
colorlinks=true, %
linkcolor=blue, %
citecolor=blue, %
filecolor=blue, %
urlcolor=blue, %
pdfauthor={Rafaella F. S. Siqueira},
pdftitle={Fractional Sobolev-Choquard critical system with Hardy term and weighted singularities}, %
pdfsubject={Partial elliptic differential equations}, %
pdfkeywords={Systems of fractional \texorpdfstring{$p$}-Laplacian, 
Doubly critical problems, 
Variational methods,
Hardy-Sobolev critical exponent, 
Choquard critical exponent}, %
pdfproducer={Latex}, %
pdfcreator={ps2pdf}}

\usepackage[hyperpageref]{backref}

\DeclareMathOperator*{\loc}{loc}

\allowdisplaybreaks

\begin{document}

\begin{abstract}
    In this paper, we consider a fractional $p$-Laplacian system of equations in the entire space $\mathbb{R}^{N}$ with doubly critical singular nonlinearities involving a local critical Sobolev  term together with a nonlocal Choquard critical term; the problem also includes a homogeneous singular Hardy term; moreover, all the nonlinearities involve singular critical weights. To prove the main result we use a refinement of Sobolev inequality
that is related to Morrey space because our problem involves doubly critical exponentss and a version of the Caffarelli-Kohn-Nirenberg inequality. 
With the help of these results, we provide sufficient conditions under which a weak nontrivial solution to the problem exists via variational methods.
\end{abstract}

\maketitle
\tableofcontents

\section{Introduction and main results}
In the present work, we consider the following fractional $p$-Laplacian system of equations in the entire space $\mathbb{R}^{N}$ featuring doubly critical nonlinearities, involving a local critical Sobolev term together with a nonlocal Choquard critical term; the problem also includes a homogeneous Hardy term; additionally, all terms have critical singular weights.
More precisely, we deal with the problem
\begin{align}
\label{prob:sistemas}
    \begin{cases}
  (-\Delta)^{s}_{p,\theta} u 
   -\gamma_1 \dfrac{|u|^{p-2}u}{|x|^{sp+ \theta}} = \left[
       I_{\mu} \ast F(\cdot, u) 
       \right](x)f(x,u) +
\dfrac{|u|^{p^*_s(\beta,\theta)-2}u} {|x|^{\beta}} 
     + \dfrac{\eta a}{a+b}\dfrac{|u|^{a-2}u|v|^b}{|x|^{\beta}}\\
     (-\Delta)^{s}_{p,\theta} v 
   -\gamma_2 \dfrac{|v|^{p-2}v}{|x|^{sp+ \theta}}
   = \left[
   I_{\mu} \ast F(\cdot, v) 
   \right](x)f(x,v)+\dfrac{|v|^{p^*_s(\beta,\theta)-2}v}{|x|^{\beta}} 
     + \dfrac{\eta b}{a+b}\dfrac{|u|^{a}|v|^{b-2}v}{|x|^{\beta}}
    \end{cases}
\end{align}
where 
$0 < s  < 1$; 
$0 < \alpha, \,\beta < sp + \theta < N$; 
$0 < \mu < N$;
$2\delta + \mu < N$; $\eta \in \mathbb{R}^+$;
$\gamma_{1}, \gamma_{2} < \gamma_{H}$
with the best fractional Hardy constant $\gamma_{H}$ 
to be defined below (without lost of generality, to simplify the notation we can consider the only parameter $\gamma=\gamma_{1}=\gamma_{2}$);
the Hardy-Sobolev and Stein-Weiss upper critical fractional exponents (this latter also called Hardy-Littlewood-Sobolev upper critical exponent) are
respectively defined by
\begin{alignat*}{2}
p^*_s(\beta,\theta)
& = \frac{p(N-\beta)}{N-sp-\theta}
& \qquad\textup{and}\qquad
p^\sharp_s(\delta,\theta,\mu)
& = \frac{p(N-\delta-\mu/2)}{N-sp-\theta};
\end{alignat*}
to simplify the notation, we write  $p^*_s(\beta,\theta)=p^*_s$ and $p^\sharp_s(\delta,\theta,\mu)=p^\sharp_s$, 
additionally $a+b=p_s^*(\beta,\theta)$.
Moreover,
$I_{\mu}(x) =|x|^{-\mu}$ is the Riesz potential of order $\mu$; 
the functions 
$f, 
 F \colon \mathbb{R}^{N} \times \mathbb{R} \to \mathbb{R}$ are respectively defined by
\begin{alignat}{2}
\label{def:ff}
f(x,t)=f_{\delta,\theta,\mu}(x,t)
& = \frac{|t|^{p^{\sharp}_{s}-2}t}%
    {|x|^{\delta}}
& \qquad\textup{and}\qquad
F(x,t)=F_{\delta,\theta,\mu}(x,t)
& = \frac{|t|^{p^{\sharp}_{s}}}%
         {|x|^{\delta}},
\end{alignat}
that is, $F_{\delta,\theta,\mu}(x,t)=p^{\sharp}_{s}
\int_{0}^{|t|}
f_{\delta,\theta,\mu}(x,\tau)\dd{\tau}$;
and the term with convolution integral, 
\begin{align*}
[I_{\mu}\ast F_{\delta,\theta,\mu}(\cdot,u)](x)
& \coloneqq \int_{\mathbb{R}^{N}} \dfrac{|u(y)|^{p^\sharp_s}}{|x-y|^{\mu}|y|^{\delta}}\dd{y},
\end{align*}
is known as Choquard type nonlinearity.

Let us now introduce the spaces of functions that are meaningful to our considerations. Throughout this work, 
we denote the norm of the weighted Lebesgue space
$L^{p}(\mathbb{R}^{N},|x|^{-\lambda})$ by
\begin{align*}
\|u\|_{L^{p}(\mathbb{R}^{N};|x|^{-\lambda})}
& \coloneqq
 \Bigl(
\int_{\mathbb{R}^{N}} \dfrac{|u|^{p}}{|x|^{\lambda}}\dd{x} \Bigr)^{\frac{1}{p}}
\end{align*}
for any 
$0\leqslant \lambda < N$ and $1 \leqslant p < +\infty$. 

We say that a Lebesgue measurable function 
$u \colon \mathbb{R}^{N} \to \mathbb{R}$ 
belongs to the weighted Morrey space 
$L_{M}^{p,\gamma+\lambda}(\mathbb{R}^{N},|x|^{-\lambda})$ 
if 
\begin{align*}
\|u\|_{L_{M}^{p,\gamma+\lambda}(\mathbb{R}^{n},|x|^{-\lambda})}
& \coloneqq
\sup_{x\in\mathbb{R}^{N},\: R\in \mathbb{R}_{+}}
\Bigl\{
\Bigl(
R^{\gamma+\lambda-N}
\int_{B_{R}(x)} \dfrac{|u|^{p}}{|x|^{\lambda}}\dd{x}
\Bigr)^{\frac{1}{p}} \Bigr\} 
< +\infty,
\end{align*}
where $1 \leqslant p < +\infty$;
$\gamma, \, \lambda \in \mathbb{R}_{+}$,
and
$0 < \gamma+\lambda < N$.

Our concerns involve the homogeneous fractional Sobolev-Slobodeckij space $\dot{W}^{s,p}_{\theta}(\mathbb{R}^N)$ 
defined as the completion of the space 
$C_{0}^{\infty}(\mathbb{R}^{N})$
with respect to the Gagliardo seminorm given by
\begin{align*}
u \mapsto
[u]_{\dot{W}^{s,p}_{\theta}(\mathbb{R}^N)}
& \coloneqq
\Bigl(
 \iint_{\mathbb{R}^{2N}} 
\frac{|u(x)-u(y)|^{p}}%
{|x|^{\theta _1}|x-y|^{N+sp}|y|^{\theta _2}}
\dd x \dd y
\Bigr)^{\frac{1}{p}},
\end{align*}
\ie
$\dot{W}^{s,p}_{\theta}(\mathbb{R}^N) = 
\overline{C_{0}^{\infty}(\mathbb{R}^{N})}^{[\,\cdot\,]} $.
We can equip the homogeneous fractional Sobolev space 
$\dot{W}^{s,p}_{\theta}(\mathbb{R}^N)$ 
with the norm 
\begin{align*}
\|u \|_{\dot{W}_{\theta}^{s,p} (\mathbb{R}^N)} 
& \coloneqq
\Bigl(
 \iint_{\mathbb{R}^{2N}} 
\dfrac{|u(x)-u(y)|^{p}}%
{|x|^{\theta _1}|x-y|^{N+sp}|y|^{\theta _2}}
\dd x \dd y
- \gamma
\int_{\mathbb{R}^N}  
\dfrac{|u|^{p}}{|x|^{sp+\theta}}
\dd x
\Bigr)^{\frac{1}{p}}\\
& \coloneqq \Big(
[u]_{\dot{W}_{\theta}^{s,p} (\mathbb{R}^n)}^{p}
-\gamma \|u \|_{L^p (\mathbb{R}^N; |x|^{-sp-\theta})}^{p}
\Bigr)^{\frac{1}{p}}.
\end{align*} 
Here, we assume that $\gamma < \gamma_{H}$, where the
best fractional Hardy constant is defined by
 \begin{align*}
\gamma_{H}
& \coloneqq 
\inf_{\substack{u \in \dot{W}^{s,p}_{\theta}(\mathbb{R}^n) \\ u \neq 0}}
\dfrac{[u]_{\dot{W}^{s,p}_{\theta}(\mathbb{R}^n)}^{p}}%
{\|u\|_{L^{p}(\mathbb{R}^{n};|x|^{-sp-\theta})}^{p}}.
\end{align*}
This turns the space 
$\dot{W}_{\theta}^{s,p} (\mathbb{R}^N)$
into a Banach space; moreover, this space is uniformly convex; in particular, it is reflexive and separable.

For simplicity, hereafter we denote the Cartesian product space of two Banach spaces $W=\dot{W}^{s,p}_{\theta}(\mathbb{R}^N)
\times \dot{W}^{s,p}_{\theta}(\mathbb{R}^N)$, 
endowed with the norm
\begin{align*}
\|(u,v)\|_{W} 
& \coloneqq
\Big( \|u \|_{\dot{W}_{\theta}^{s,p} (\mathbb{R}^N)}^{p} 
+ \|v \|_{\dot{W}_{\theta}^{s,p} (\mathbb{R}^N)}^{p} \Big)^{1/p}.
\end{align*}

Intuitively, problem~\eqref{prob:sistemas} is 
understood as showing the existence of a pair 
$(u,v) \in W$ such that
\begin{align*}
    \begin{cases}
  (-\Delta)^{s}_{p,\theta} u 
   -\gamma \dfrac{|u|^{p-2}u}{|x|^{sp+ \theta}}
   = \Bigl(\displaystyle\int_{\mathbb{R}^N}\frac{|u|^{p^{\sharp}_s}}{|x-y|^{\mu}|y|^{\delta}}\dd y\Bigr)\frac{|u|^{p^{\sharp}_s-2}u}{|x|^{\delta}} +
\dfrac{|u|^{p^*_s -2}u} {|x|^{\beta}} 
     + \dfrac{\eta a}{a+b}\dfrac{|u|^{a-2}u|v|^b}{|x|^{\beta}}\\
     (-\Delta)^{s}_{p,\theta} v 
   -\gamma \dfrac{|v|^{p-2}v}{|x|^{sp+ \theta}} =\Bigl(\displaystyle\int_{\mathbb{R}^N}\frac{|v|^{p^{\sharp}_s}}{|x-y|^{\mu}|y|^{\delta}}\dd y\Bigr)\frac{|v|^{p^{\sharp}_s-2}v}{|x|^{\delta}}+\dfrac{|v|^{p^*_s-2}v}{|x|^{\beta}} 
     + \dfrac{\eta b}{a+b}\dfrac{|u|^{a}|v|^{b-2}v}{|x|^{\beta}}
    \end{cases}
\end{align*}
where the fractional $p$-Laplacian operator is defined 
for $\theta=\theta_{1}+\theta_{2}$, $x\in\mathbb{R}^{N}$, 
and any function $u \in C_{0}^{\infty}(\mathbb{R}^{N})$, as
\begin{align}
\label{plaplacianofracionario}
(-\Delta)_{p,\theta}^{s}u(x)
&\coloneqq
\textup{p.v.\xspace}
\int_{\mathbb{R}^{N}} 
\dfrac{|u(x)-u(y)|^{p-2}(u(x)-u(y))}{|x|^{\theta_{1}}
|x-y|^{N+sp}|y|^{\theta_{2}}}\dd{y},
\end{align} 
and p.v.\xspace\ is the
Cauchy's principal value. This operator is the prototype of nonlinear nonlocal elliptic operator and can also be defined on smooth functions by 
\begin{align*}
    (-\Delta)_{p,\theta}^{s}u(x)
&\coloneqq
2\lim_{\varepsilon\to 0} 
\int_{\mathbb{R}^{N}\backslash B_{\varepsilon}(x)} 
\dfrac{|u(x)-u(y)|^{p-2}(u(x)-u(y))}{|x|^{\theta_{1}}
|x-y|^{N+sp}|y|^{\theta_{2}}}\dd{y}.
\end{align*}
This definition is consistent, up to a normalization constant $C=C(N,s,\theta)$, with the usual definition of the linear fractional Laplacian operator $(-\Delta)^{s}$ for $p=2$ and $\theta=0$.

Our main goal in this work is to show that 
problem~\eqref{prob:sistemas} admits at least a weak solution, by which term we mean a function 
$(u,v) \in W$ such that
\begin{align*}
   \lefteqn{\iint_{\mathbb{R}^{2N}} \frac{|u(x)-u(y)|^{p-2}(u(x)-u(y))(\phi_1(x)-\phi_1(y))}{|x|^{\theta_1}|x-y|^{N+sp}|y|^{\theta_2}}\dd x \dd y} \\
    & \lefteqn{+ \iint_{\mathbb{R}^{2N}} \frac{|v(x)-v(y)|^{p-2}(v(x)-v(y))(\phi_2(x)-\phi_2(y))}{|x|^{\theta_1}|x-y|^{N+sp}|y|^{\theta_2}}\dd x \dd y } \\
    \lefteqn{-\gamma_1 \int_{\mathbb{R}^N} \frac{|u|^{p-2}u\phi_1}{|x|^{sp+\theta}}\dd x -
    \gamma_2 \int_{\mathbb{R}^N} \frac{|v|^{p-2}v\phi_2}{|x|^{sp+\theta}}\dd x} \\
    & \quad = \iint_{\mathbb{R}^{2N}} \frac{|u(x)|^{p_s^{\sharp}-2}|u(y)|^{p_s^{\sharp}}u(x)\phi_1(x)}{|x|^{\delta}|x-y|^{\mu}|y|^{\delta}} \dd x \dd y \\
    & \qquad + \iint_{\mathbb{R}^{2N}} \frac{|v(x)|^{p_s^{\sharp}-2}|v(y)|^{p_s^{\sharp}}v(x)\phi_2(x)}{|x|^{\delta}|x-y|^{\mu}|y|^{\delta}} \dd x \dd y \\
    & \qquad +  \int_{\mathbb{R}^N}\frac{|u|^{p_s^{*}(\beta, \theta)-2}u\phi_1(x)}{|x|^{\beta}} \dd x + \int_{\mathbb{R}^N}\frac{|v|^{p_s^{*}(\beta, \theta)-2}v\phi_2(x)}{|x|^{\beta}} \dd x \\
    & \qquad +  \int_{\mathbb{R}^N} \frac{\eta a |u|^{a-2}u\phi_1|v|^{b}}{|x|^{\beta}} \dd x +  \int_{\mathbb{R}^N} \frac{\eta b|u|^a|v|^{b-2}v\phi_2}{|x|^{\beta}} \dd x 
\end{align*}
for any pair  of test functions $(\phi_1, \phi_2) \in W$. 
Now we define the energy functional 
$I \colon W \to \mathbb{R}$ by
\begin{align}
    \label{funcionaldeenergia}
    I(u,v) & 
    = \frac{1}{p} \bigl[\iint_{\mathbb{R}^{2N}} \frac{|u(x)-u(y)|^p}{|x|^{\theta_1}|x-y|^{N+sp}|y|^{\theta_2}} \dd x \dd y + \iint_{\mathbb{R}^{2N}} \frac{|v(x)-v(y)|^p}{|x|^{\theta_1|}x-y|^{N+sp}|y|^{\theta_2}} \dd x \dd y \bigr] \nonumber\\
    &\quad - \frac{\gamma_1}{p} \int_{\mathbb{R}^N} \frac{|u|^{p}}{|x|^{sp+\theta}}\dd x - \frac{\gamma_2}{p} \int_{\mathbb{R}^N} \frac{|v|^{p}}{|x|^{sp+\theta}}\dd x \nonumber \\
    & \quad - \frac{1}{2p_s^{\sharp}(\delta, \theta, \mu)} \bigl[\iint_{\mathbb{R}^{2N}}\frac{|u(x)|^{p_s^{\sharp}}|u(y)|^{p_s^{\sharp}}}{|x|^{\delta}|x-y|^{\mu}|y|^{\delta}}\dd x \dd y +  \iint_{\mathbb{R}^{2N}}\frac{|v(x)|^{p_s^{\sharp}}|v(y)|^{p_s^{\sharp}}}{|x|^{\delta}|x-y|^{\mu}|y|^{\delta}}\dd x \dd y \bigr] \nonumber\\
    & \quad - \frac{1}{p^*_s}\bigl[\int_{\mathbb{R}^N}\frac{|u|^{p^*_s}}{|x|^{\beta}} \dd x + \int_{\mathbb{R}^N}\frac{|v|^{p^*_s}}{|x|^{\beta}} \dd x\bigr] - \frac{1}{p^*_s} \int_{\mathbb{R}^N} \frac{\eta |u|^a|v|^b}{|x|^{\beta}} \dd x.
\end{align}

For the parameters in the previously specified intervals, the energy functional $I$ is well defined and is continuously differentiable, 
\ie $I \in C^{1}(\dot{W}^{s,p}_{\theta}(\mathbb{R}^N);\mathbb{R})$; moreover, a nontrivial critical point of the energy functional $I$ is a nontrivial weak solution to problem~\eqref{prob:sistemas}.

\begin{theorem}
\label{teo:sistemas}
Problem~\eqref{prob:sistemas} has at least a nontrivial weak solution provided that 
$0 < s < 1$;
$0 < \alpha, \, \beta < sp + \theta < N$; 
$0 < \mu < N$; $a+b=p^*_s(\beta,\theta)$; $\eta \in \mathbb{R}^+$ and $\gamma_1,\gamma_2 < \gamma _H$.
\end{theorem}

In this work we also consider the following variants of problem~\eqref{prob:sistemas}, namely one problem with a Hardy potential and double Sobolev type nonlinearities,
\begin{align}
\label{prob:sistemas02}
    \begin{cases}
  (-\Delta)^{s}_{p,\theta} u 
   -\gamma_1 \dfrac{|u|^{p-2}u}{|x|^{sp+ \theta}} 
   = \displaystyle\sum_{k=1}^{k=2} \dfrac{|u|^{p^*_s-2}u} {|x|^{\beta}} 
     + \dfrac{\eta a}{a+b}\dfrac{|u|^{a-2}u|v|^b}{|x|^{\beta_k}}\\
     (-\Delta)^{s}_{p,\theta} v 
   -\gamma_2 \dfrac{|v|^{p-2}v}{|x|^{sp+ \theta}} 
   = \displaystyle\sum_{k=1}^{k=2}\dfrac{|v|^{p^*_s-2}v}{|x|^{\beta_k}} 
     + \dfrac{\eta b}{a+b}\dfrac{|u|^{a}|v|^{b-2}v}{|x|^{\beta_k}}
    \end{cases}
\end{align}
and another one with a Hardy potential and double Choquard type nonlinearities,
\begin{align}
\label{prob:sistemas03}
    \begin{cases}
  (-\Delta)^{s}_{p,\theta} u 
   -\gamma_1 \dfrac{|u|^{p-2}u}{|x|^{sp+ \theta}} 
   = \displaystyle\sum_{k=1}^{k=2}\left[
       I_{\mu_k} \ast F_{\delta,\theta,\mu_k}(\cdot, u) 
       \right](x)f_{\delta,\theta,\mu_k}(x,u) 
     + \dfrac{\eta a}{a+b}\dfrac{|u|^{a-2}u|v|^b}{|x|^{\beta_k}}\\
     (-\Delta)^{s}_{p,\theta} v 
   -\gamma_2 \dfrac{|v|^{p-2}v}{|x|^{sp+ \theta}} 
   = \displaystyle\sum_{k=1}^{k=2}\left[
       I_{\mu_k} \ast F_{\delta,\theta,\mu_k}(\cdot, v) 
       \right](x)f_{\delta,\theta,\mu_k}(x,v) 
     + \dfrac{\eta b}{a+b}\dfrac{|u|^{a}|v|^{b-2}v}{|x|^{\beta_k}}
    \end{cases}
\end{align}

The notion of weak solution to problems~\eqref{prob:sistemas02}  
and~\eqref{prob:sistemas03} 
can be defined in the same way as that for problem~\eqref{prob:sistemas}, \ie
we multiply the differential equations by a pair of test functions and use a kind of integration by parts. Then we recognize these expressions as the derivatives of an energy functional which, under the appropriate hypotheses on the parameters, is continuously differentiable. This means that weak solutions to these problems are critical points of the appropriate energy functional. By adaptting the method used in the proof of Theorem~\ref{teo:sistemas} we deduce the following result.
\begin{theorem}
\label {teo:sistemas02}
Problems~\eqref{prob:sistemas02} and~\eqref{prob:sistemas03} have at least a nontrivial weak solution
under similar assumptions as in Theorem~\ref{teo:sistemas}, \ie
$0 < s < 1$;
$0 < \alpha_k, \, \beta_k < sp + \theta < N$; 
$0 < \mu_k < N$; $a+b=p^*_s(\beta_k,\theta)$; $\eta \in \mathbb{R}^+$ and $\gamma_1,\gamma_2 < \gamma _H$ for $k \in \{1,2\}$.
\end{theorem}

\subsubsection*{Notation.}
For $\rho \in \mathbb{R}_{+}$, we define 
$B_{\rho}(x)\coloneqq
\{y\in\mathbb{R}^{N} \colon |x-y|<\rho\}$, the open ball centered at $x$ with radius $\rho$. The constant $\omega_{N}$ denotes the volume of the unit ball in $\mathbb{R}^{N}$.
The arrows $\to$ and $\rightharpoonup$ denote the strong convergence and the weak convergence, respectively.
Given the functions $f,g \colon \mathbb{R}^{N}\to\mathbb{R}$, we recall that $f=O(g)$ if there is a constant $C\in \mathbb{R}_{+}$ such that 
$|f(x)|\leqslant C|g(x)|$ for all
$x \in \mathbb{R}^{N}$;
and $f=o(g)$ as $x\to x_{0}$ if 
$\lim_{x\to x_{0}} |f(x)|/|g(x)|=0$.
The pair $r$ and $r'$ denote H\"{o}lder conjugate exponents, \ie 
$1/r + 1/r'=1$ or $r+r'=r r'$.
The positive and negative parts of a function $\phi$ are denoted by $\phi_{\pm}\coloneqq \max\{\pm\phi,0\}$. Throughout this paper, we will use the following notations: $tz \coloneqq t(u, v) = (tu, tv)$ for all
$(u, v) \in W$ and $t \in \mathbb{R}$; $(u, v)$ is said to be nonnegative in $\mathbb{R}^N$ if $u \geqslant 0$ and $v \geqslant 0$ in $\mathbb{R}^N$; $(u, v)$ is
said to be positive in $\mathbb{R}^N$ if $u > 0$ and $v > 0$ in $\mathbb{R}^N$. Finally, $C\in\mathbb{R}_{+}$ denotes a universal constant that may change from line to line; when it is relevant, we will add subscripts to specify the dependence of certain parameters.

\section{Historical background}

\subsubsection*{The fractional Laplacian}
There are many equivalent definitions of the fractional Laplacian. In our case, on the Euclidean space $\mathbb{R}^{N}$ of dimension $N \geqslant 1$, for $\theta=\theta_{1}+\theta_{2}$ and the above specified intervals for the parameters, we define the nonlocal elliptic $p$-Laplacian operator with the help of the Cauchy's principal value integral as in~\eqref{plaplacianofracionario}.

For problems with two nonlinearities involving the Laplacian operator, see Filipucci, Pucci \& Robert~\cite{filippucci2009p}. For similar problems involving the fractional Laplacian, see 
Servadei \& Valdinoci~\cite{servadei2012mountain},
Ghoussoub \& Shakerian~\cite{MR3366777}, 
Chen \& Squassina~\cite{chen2016critical}
Chen~\cite{MR3762809,chen2019existence},
Assun{\c{c}}{\~a}o, Silva \& Miyagaki~\cite{assunccao2020fractional}.
For a survey paper on the subject of fractional Sobolev spaces, see Di Nezza, Palatucci \& Valdinoci~\cite{MR2944369}; see also
Molica Bisci, R\u{a}dulescu \& Servadei~\cite{MR3445279}.

\subsubsection*{The Choquard equation}
On the Euclidean space $\mathbb{R}^{N}$ of dimension $N \geqslant 1$ and for $x\in \mathbb{R}^{N}$, the equation
$-\Delta u + V(x) u = (I_{\mu} \ast |u|^{q}) |u|^{q-2}u$
was introduced by Choquard in the case $N=3$ and $q=2$ to model a certain approximation to Hartree-Fock theory of one-component plasma and to describe a electron trapped in its own hole. Also in this situation it finds physical significance in the work by Fr\"{o}lich and Pekar on the description of the quantum mechanics of a polaron at rest. When $V(x) \equiv 1$, the groundstate solutions exist if $2^{\flat}\coloneqq 2(N-\mu/2))/N < q < 2(N-\mu/2)/(N-2s) \coloneqq 2^{\sharp}$ due to the mountain
pass lemma or the method of the Nehari manifold, 
while there are no nontrivial solution if $q = 2^{\flat}$ or if $q = 2^{\sharp}$ as
a consequence of the Pohozaev identity.

In general, the associated Schr\"{o}dinger-type evolution equation
$i \partial_{t} \psi 
= \Delta \psi
+ \big( I_{\mu} \ast |\psi|^{2} \big) \psi$
is a model for large systems of atoms with an attractive interaction that is weaker and has a longer range than that of the nonlinear Schr\"{o}dinger equation. 
Standing wave solutions of this equation
are solutions to the Choquard equation. 
For more information on the various results related to the non-fractional
Choquard-type equations and their variants see Moroz \& Van Schaftingen~\cite{MR3625092} and 
Mukherjee \& Sreenadh~\cite{MR3720946}.

\subsubsection*{The Morrey spaces}
The study of Morrey spaces is motivated by many reasons.
Initially, these
spaces were introduced by Morrey in order to understand the regularity of solutions to elliptic partial
differential equations. 
Regularity theorems, which allow one to conclude higher regularity of a function that is a solution of a differential equation together with a lower regularity of that function, play a central role in the theory of partial differential equations. 
One example of this kind of regularity theorem is a version of the Sobolev embedding theorem which states that 
$W^{j+m,p}(\Omega) \subset C^{j,\lambda}(\overline{\Omega})$ for $0 < \lambda \leqslant m - N/p$, where $j \in \mathbb{N}$ and $\Omega \subset \mathbb{R}^{N}$ is a Lipschitz domain.

Morrey spaces can complement the boundedness properties of operators that Lebesgue spaces can not handle. In line with this, many authors study the boundedness of various
integral operators on Morrey spaces. The theory of Morrey spaces may come in useful when the Sobolev embedding theorem is not readily available. For more information on Morrey spaces, see Gantumur~\cite{gantumur} and Sawano~\cite{MR4038542}.

\subsubsection*{Systems of fractional elliptic equations}
The subject of two or more fractional elliptic equations have been widely studied in recent years.
We devote this section on briefly glimpsing the results that have already been proved in the context of existence, non-existence, uniqueness and multiplicity of solutions to systems of fractional elliptic equations. 

Liu \& Wang~\cite{liu2010ground} gave a sufficient condition on large coupling coefficients for the existence of a nontrivial ground state solution in a system of nonlinear Schrödinger equation; they also considered bound state solutions. 
Chen \& Deng~\cite{chen2016nehari} investigated the existence of two nontrivial solutions to the fractional $p$-Laplacian system involving concave-convex nonlinearities via the Nehari method. 
Chen~\cite{chen2016infinitely} obtained the existence of infinitely many nonnegative solutions for a class of the quasilinear Schrödinger system in $\mathbb{R}^N$ in the Laplacian setting and investigate the multiplicity of solutions for a $p$-Kirchhoff
system driven by a nonlocal integro-differential operator with zero Dirichlet boundary data. Xiang, Zhang \& R\u{a}dulescu~\cite{xiang2016multiplicity} studied the multiplicity of solutions for a $p$--Kirchhoff system driven by a nonlocal integro-differential operator with zero Dirichlet boundary data. 
Chen \& Squassina~\cite{chen2016critical} used Nehari manifold techniques to obtain the existence of multiple solutions to a fractional $p$--Laplacian system involving critical concave-convex nonlinearities. Fiscella, Pucci \& Saldi~\cite{fiscella2017existence}, using several variational methods, dealt with the existence of nontrivial nonnegative solutions of
Schrödinger–Hardy systems driven by two possibly different fractional $p$--Laplacian
operators. The main features of this paper is the presence of the Hardy term and the fact that the nonlinearities do not necessarily satisfy the Ambrosetti–Rabinowitz condition. Wang, Zhang \& Zhang~\cite{wang2017fractional} are interested in a fractional Laplacian system in the whole space $\mathbb{R}^N$, which involves critical Sobolev-type nonlinearities and critical Hardy-Sobolev-type nonlinearities. 
Yang~\cite{yang2021doubly}  considered the existence of nontrivial weak solutions to a doubly critical system involving fractional Laplacian in $\mathbb{R}^N$ with subcritical weight. More recently,
Lu \& Shen~\cite{MR4129343} studied a critical fractional $p$-Laplacian system with homogeneous nonlinearity; they used a concentration compactness principle associated with fractional $p$-Laplacian system for the fractional order Sobolev spaces in bounded domains, which is significantly more difficult to prove than in the case of a single fractional $p$-Laplacian equation and is of independent interest.

Most of the existing results have been developed for systems with two equations.
For a general system, Lin \& Wei~\cite{lin2005ground,MR2358296} studied a system with several coupled nonlinear Schr\"{o}dinger equations in the whole space up to three dimensions which has some applications in nonlinear optics. The existence of ground state solutions may depend on the coupling constants that model the interaction between the components of the system. If a constant is positive, the interaction is attractive; otherwise, it is repulsive. When all the constants are positive and some associated matrix is positively definite, they proved the existence of a radially symmetric ground state solution; however, if all the constants are negative, our if one of them is negative and the matrix is positively definite, there is no ground state solution. They also obtained the existence of a bound state solution which is non-radially symmetric in the three dimensional case.

\subsubsection*{Our contribution to the problem and some of its difficulties}
The present work is motivated by Assunção, Miyagaki \& Siqueira~\cite{assunccao2023fractional}, Wang, Zhang \& Zhang~\cite{wang2017fractional} and Yang~\cite{yang2021doubly}. 
Our existence result can be regarded as an extension and improvement of the corresponding existence results in these works. More precisely, we will extend the result in~\cite{assunccao2023fractional} to a system of coupled equations in $\mathbb{R}^N$ with the general fractional $p$-Laplacian
with $p > 1$ and $\theta = \theta_1 + \theta_2$ not necessarily zero. Moreover, we use a refinement of Sobolev inequality that is related to Morrey space because our problem involves doubly critical exponents. As one can expect, the nonlocality of the fractional $p$-Laplacian makes it more difficult to study. In our case, one of the main difficulties when dealing with this problem is the lack of compactness of Sobolev embedding theorem for the critical exponent. Therefore, we have to develop a precise analysis of the level of the Palais-Smale sequences obtained with the application of the mountain pass theorem and study their behavior concerning strong convergence of one of its scaled subsequences.

\section{Existence of solutions for auxiliary minimization problems}
We begin this section by introducing two important and sharp Rayleigh-Ritz constants. The first one is related to the Gagliardo seminorm, the Hardy term, and a convolution involving the upper Stein-Weiss exponent,
\begin{align}
\label{Ssharp:sistemas}
S^\sharp
&\coloneqq
\inf_{(u,v)\in W\setminus\{0\}}
\dfrac{\|(u,v)\|_W^{p}}{(Q^{\sharp}(u,v))^{\frac{p}{2p_{s}^{\sharp}(\delta,\theta,\mu)}}},
\end{align}
where the quadratic form $Q^{\sharp} \colon W \to \mathbb{R}$ is given by
\begin{align}
\label{qsharp:sistemas}
Q^{\sharp}(u,v)
& \coloneqq
\iint_{\mathbb{R}^{2N}}
\dfrac{|u(x)|^{p_{s}^{\sharp}(\delta,\theta,\mu)}
       |u(y)|^{p_{s}^{\sharp}(\delta,\theta,\mu)}+|v(x)|^{p_{s}^{\sharp}(\delta,\theta,\mu)}
       |v(y)|^{p_{s}^{\sharp}(\delta,\theta,\mu)}}%
       {|x|^{\delta}
        |x-y|^{\mu}
        |y|^{\delta}}
        \dd{x}\dd{y}.
\end{align}

The second one is related to the Gagliardo seminorm, ther Hardy term, and the norm in the critical weighted Lebesgue space, that is, the Sobolev constant,
\begin{align}
\label{Sast:sistemas}
S^{\ast}
&\coloneqq
\inf_{(u,v)\in W\setminus\{0\}}
\dfrac{\|(u,v)\|_W^p}{(Q^{*}(u,v))^{\frac{p}{p^*_s(\beta,\theta)}}},
\end{align}
where the quadratic form $Q^{*}\colon W \to \mathbb{R}$ is given by
\begin{align}
\label{qestrela:sistemas}
Q^{*}(u,v)
&\coloneqq \int_{\mathbb{R}^N}\frac{|u(x)|^{p^*_s(\beta,\theta)}+|v(x)|^{p^*_s(\beta,\theta)}+\eta |u(x)|^a|v(x)|^b}{|x|^{\beta}} \dd{x}   
\end{align}
For general $p\neq 2$, the explicit formula for the extremal functions for the $p$-fractional Sobolev inequality is not known yet, though it is conjectured that it is of the form
\begin{align*}
    U(x)=\dfrac{C}{(1+|x|^{\frac{p}{p-1}})^{\frac{N-sp}{p}}}
\end{align*} 
up to translation and scaling. However, there is a result about the asymptotic behavior of $U$, as seen in Brasco, Mosconi \& Squassina~\cite{brasco2016optimal} and Mosconi, Perera, Squassina \& Yang~\cite{MR3530213}.

One of the first major difficulties that we encounter is the lack of an explicit formula for a minimizer of the quantity $S^{\ast}$. There is a conjecture about the minimizers which states that they have the form 
\begin{align*}
U(x) & = \dfrac{c}{[1+(|x-x_{0}|/\varepsilon))^{\frac{p}{p-1}}]^{\frac{N-sp}{p}}}
\end{align*}
where $C \neq \mathbb{R}^{N}$, 
$x_{0} \in \mathbb{R}^{N}$, and
$\varepsilon \in \mathbb{R}_{+}$.
This conjecture was proved by Lieb~\cite{MR717827} in the case $p=2$; however, for $p \neq 2$ it is not even known if these functions are minimizers. To overcome this difficulty, we will work with some asymptotic estimates for minimizers recently obtained by Brasco, Mosconi \& Squassina~\cite{MR3461371};
see also Mosconi, Perera, Squassina \& Yang~\cite{MR3530213}.

In the next lemma we present some embeddings of Morrey spaces into weighted Lebesgue spaces; the proof is straightforward.
\begin{lemma}
\label{lemma:imersoes}
The following fundamental properties are true.
\begin{enumerate}
    \item $L^{p\rho}(\mathbb{R}^N, |y|^{-\rho \lambda}) \hookrightarrow L^{p,\gamma +\lambda}(\mathbb{R}^N, |y|^{-\lambda})$ for $\rho = \frac{N}{\gamma +\lambda}>1$.

    \item For any $p \in (1, + \infty)$, we have $L^{p, \gamma + \lambda}(\mathbb{R}^N, |y|^{- \lambda}) \hookrightarrow L^{1,\frac{\gamma}{p} +\frac{\lambda}{p}}(\mathbb{R}^N, |y|^{-\frac{\lambda}{p}})$. 

    \item For $1 \leqslant p < +\infty$ and $\gamma + \lambda = N$, we have 
    $ L_{M}^{p,N}(\mathbb{R}^{N}, |y|^{-\lambda}) =L^p(\mathbb{R}^N, |y|^{-\lambda})$,
    i.e.,\xspace   $L_{M}^{p,N}(\mathbb{R}^{N}, |y|^{-\lambda})$ and $L^p(\mathbb{R}^N, |y|^{-\lambda})$ are continuously embedded in each other. 
    \item If if we assume that $s \in (0,1)$ and $0< \alpha < sp+\theta < N $, for $1 \leqslant q <  p^*_s(\alpha,\theta)$ and
    $r= \frac{\alpha}{p^*_s(\alpha,\theta)}$, it holds
    \begin{align}
    \label{imersoes}
        \dot{W}^{s,p}_{\theta}(\mathbb{R}^N) \hookrightarrow L^{p^*_s(\alpha,\theta)}(\mathbb{R}^N, |y|^{-\alpha}) \hookrightarrow L_{M}^{q, \frac{(N-sp-\theta)q}{p}+qr}(\mathbb{R}^N, |y|^{-pr})
    \end{align}
    and the norms in these spaces share the same dilation invariance.

    \item For any $q \in [1, p^*_s(0,\theta)), \dot{W}^{s,p}_{\theta}(\mathbb{R}^N) \hookrightarrow L^{p^*_s(0,\theta)}(\mathbb{R}^N) \hookrightarrow L^{q, \frac{(N-sp-\theta)q}{p}}(\mathbb{R}^N).$
\end{enumerate}

\end{lemma}

For more properties of Lebesgue spaces, integral inequalities and boundedness properties of the operators in generalized Morrey spaces, see Sawano~\cite{MR4038542}.

Next, we use the following versions of the fractional Hardy-Sobolev and Caffarelli-Kohn-Nirenberg inequalities; see Nguyen \& Squassina~\cite[Theorem~1.1]{MR3771839}; see also Abdellaoui \& Bentifour~\cite{MR3626031}.

\begin{lemma}
\label{lemma:2.2chineses:bis}
Let $N \geqslant 1, p \in (1,+\infty), s \in (0,1), 0\leqslant \alpha \leqslant sp+\theta<N, \theta, \theta_1,\theta_2,\beta \in \mathbb{R}$ be such that $\theta = \theta _1 + \theta _2$. If $1/p^*_s(\alpha,\theta) - \alpha / Np^*_s(\alpha,\theta)>0$, then there exists a positive constant $C(N,\alpha,\theta)$ such that
\begin{align}
    \label{des:2.2chineses:bis}
    \left(\int_{\mathbb{R}^N}\frac{|u|^{p^*_s(\alpha,\theta)}}{|x|^{\alpha}}\dd x\right)^{\frac{p}{p^*_s(\alpha,\theta)}} \leqslant C(N,\alpha,\theta) \iint_{\mathbb{R}^{2N}}      \frac{|u(x)-u(y)|^p}{|x|^{\theta _1}|x-y|^{N+sp}|y|^{\theta _2}} \dd x \dd y
\end{align}
for all $u \in \Dot{W}^{s,p}_{\theta}(\mathbb{R}^N)$.
\end{lemma}

We note that the norm  $\| \cdot \|$ is comparable with the Gagliardo seminorm $[\,\cdot\,]_{\Dot{W}^{s,p}_{\theta}(\mathbb{R}^N)}$ as stated in the next result.
\begin{corollary}
Under the hypotheses of Lemma~\ref{lemma:2.2chineses:bis},
if $\gamma < \gamma _H$ then
\begin{align}
    \label{des:2.4chineses:bis}
    \begin{split}
       {}&\Bigl(1-\frac{\gamma _+}{\gamma _H}\Bigr) \iint_{\mathbb{R}^{2N}}      \frac{|u(x)-u(y)|^{p}}{|x|^{\theta _1}|x-y|^{N+sp}|y|^{\theta _2}}\dd x \dd y \\
       &\quad \leqslant \|u\|^p \leqslant \Bigl(1+\frac{\gamma_-}{\gamma_H}\Bigr)\iint_{\mathbb{R}^{2N}}      \frac{|u(x)-u(y)|^{p}}{|x|^{\theta _1}|x-y|^{N+sp}|y|^{\theta _2}}\dd x \dd y,
    \end{split}
\end{align}
where $\gamma_{\pm} = \max \{\pm \gamma,0\}$.
\end{corollary}

We can deduce another useful inequality.
\begin{corollary}
Under the hypotheses of Lemma~\ref{lemma:2.2chineses:bis} we have
\begin{align}
\label{limitacaopsharp}
    \iint_{\mathbb{R}^{2N}} \frac{|u(x)|^{p^{\sharp}_{s}(\delta,\theta,\mu)}
    |u(y)|^{p^{\sharp}_{s}(\delta,\theta,\mu)}}{|x|^{\delta}|x-y|^{\mu}|y|^{\delta}}\dd x\dd y
    & \leqslant 
    C(N,\delta,\mu)
    \|u\|^{2 p^{\sharp}_{s}(\delta,\theta,\mu)}_{\Dot{W}^{s,p}_{\theta}(\mathbb{R}^N)}.
\end{align}
\end{corollary}

Based on the embeddings~\eqref{imersoes} 
we establish an improved weighted fractional Caffarelli-Kohn-Nirenberg inequality whose proof can be found in~\cite{assunccao2023fractional}; see also~\cite{MR3216834}.

\begin{lemma}
(Fractional Caffarelli-Kohn-Nirenberg inequality)
\label{prop:1.3chineses:bis}
    Let $s\in (0,1)$ and $0< \beta < sp +\theta< N$. Then there exists $C = C(N, s, \beta)>0$ such that for any $\zeta \in (\Bar{\zeta}, 1)$ and for any $q\in [1, p^*_s)$, for all $u \in \Dot{W}^{s,p}_{\theta}(\mathbb{R}^N)$ it holds
    \begin{align}
    \label{des:1.10chineses:bis}   \biggl(\int_{\mathbb{R}^N}\frac{|u(y)|^{p^*_s}}{|y|^{\beta}}\dd y\bigg)^{\frac{1}{p^*_s}} \leqslant \|u\|_{\Dot{W}^{s,p}_{\theta}}^{\zeta} \|u\|^{1-\zeta}_{L_M^{q,\frac{N-sp-\theta}{p}q+qr}(\mathbb{R}^N, |y|^{-qr})}
    \end{align}
    where $\Bar{\zeta} = \max 
    \{p/p^*_s, (p^*_s(0,\theta)-1)/p^*_s\}>0$ and $r=\beta/p^*_s$.
    \end{lemma}

Now we state a result about local convergence; see~\cite[Lemma $2.3$]{Li_2020}.

\begin{lemma}
\label{lemma:2.3chineses}
Let $s \in (0,1)$ and $0<r<s+ \frac{\theta}{p}<\frac{N}{p}$. If $\{u_k\}$ is a bounded sequence in $\Dot{W}^{s,p}_{\theta}(\mathbb{R}^N)$ and $u_k \rightharpoonup u$ in $\Dot{W}^{s,p}_{\theta}(\mathbb{R}^N)$, then as $k\to+\infty$,
\begin{align*}
    \frac{u_k}{|x|^{r+\frac{\theta r}{sp}}} \to \frac{u}{|x|^{r+\frac{\theta r}{sp}}} \; \textup{in $L^p_{\loc}(\mathbb{R}^N)$}.
\end{align*}
\end{lemma}

\begin{proposition}
\label{minimizadores:sistemas}
    For $s \in (0,1)$ the best constants $S^{\sharp}$ and $S^*$ verify  the following items.
    \begin{enumerate}[wide]
        \item If $0<\alpha<sp+\theta<N, \mu \in (0,N)$ and $\gamma < \gamma _H$, then $S^{\sharp}$ is attained in ${W}$;
        \item If $0<\beta<sp+\theta<N$ and $\gamma < \gamma _H$, then $S^*$ is attained in ${W}$.
    \end{enumerate}
\end{proposition}

\begin{proof}
    \begin{enumerate}[wide]
        \item If $0< \alpha <sp+\theta<N$ and $\gamma < \gamma _H$, let $\{(u_k,v_k)\}_{k\in\mathbb{N}} \subset {W}$ be a minimizing sequence of $S^{\sharp}$ such that
        \begin{align}
        \label{q:igual11}
            Q^{\sharp}(u_k, v_k) = 1, \qquad \|(u_k,v_k)\|^p \to S^{\sharp}
        \end{align}
        as $k\to +\infty$.     Recall that $r=\frac{\alpha}{p^*_s(\alpha,\theta)}$. 
\begin{claim}
\label{limitacaomorey:sistemas}
We have $ C_1  \leqslant \|u_k\|_{L_{M}^{q,\frac{N-sp-\theta}{p}q+qr}(\mathbb{R}^N, |y|^{-pr})} \leqslant C_2$.
\end{claim}        
\begin{proof}
            We just use the
 embeddings~\eqref{imersoes} and the Caffarelli-Kohn-Nirenberg's inequality~\eqref{des:1.10chineses:bis}. For more details, see~\cite[Eq.~(4.2), p.~23]{assunccao2023fractional}        \end{proof}

For any $k \in \mathbb{N}$ large enough, we may find $\lambda _k >0$ and $x_k \in \mathbb{R}^N$ such that
\begin{align}
    \lambda _k ^{-sp-\theta +pr} \int_{B_{\lambda _{k} (x_k)}} \frac{|u_k(y)|^p}{|y|^{pr}}\dd y > \|u_k\|^p_{L_{M}^{p,N-sp-\theta+pr}(\mathbb{R}^N, |y|^{-pr})} -\frac{C}{2k} \geqslant C >0
\end{align}
for constants $C \in\mathbb{R}_{+}$.

Our goal is to pass to the limit as $k \to +\infty$ in the minimizing sequence. To do this, we adapt the Levy's concentration principle; more precisely, we create another sequence $\{(\tilde{u}_k,\tilde{v}_k)\}_{k\in\mathbb{N}} \subset {W}$
that will help us to control the radius and the centers of these balls. Let 
\begin{align*}
\Tilde{u}_k (x) =\lambda _{k} ^{\frac{N-sp-\theta}{p}}u_k(\lambda _{k} x) \quad \text{and} \quad \Tilde{v}_k (x) =\lambda _{k} ^{\frac{N-sp-\theta}{p}}v_k(\lambda _{k} x)
\end{align*} 
be the appropriate scaling for the class of problems that we consider and define $\Tilde{x}_{k} \coloneqq x/\lambda _{k}$ and $\Tilde{x}_{k} \coloneqq x/\lambda _{k}$. Then, using the change of variables $y=\lambda_{k}x$ with $\dd{y}=\lambda_{k}^{N}\dd{x}$ we have,
\begin{align}    \label{des:4.1chineses:sistemas}
    \int_{B_1\left(\frac{x_k}{\lambda _{k}}\right)} \frac{|\Tilde{u}_k(x)|^p}{|x|^{pr}}\dd x  \geqslant C>0 \quad \text{and} \quad    \int_{B_1\left(\frac{x_k}{\lambda _{k}}\right)} \frac{|\Tilde{v}_k(x)|^p}{|x|^{pr}}\dd x  \geqslant C>0.
\end{align}

Now we claim that $S^{\sharp}$ is invariant under the previously defined dilation.

In fact, $ Q^{\sharp}(\Tilde{u}_k,\Tilde{v}_k) = 1$. To show this property, we use the change of variables $\Bar{x}=\lambda _{k} x$ and $\Bar{y}=\lambda _{k} y$, we have 
\begin{align*}
Q^{\sharp}(\Tilde{u}_k,\Tilde{v}_k)
& = 
\iint_{\mathbb{R}^{2N}}\frac{|u_k(x)|^{p^\sharp_s}|u_k(y)|^{p^\sharp_s}}{|x|^{\delta }|x-y|^{\mu}|y|^{\delta}} \dd x \dd y
+ \iint_{\mathbb{R}^{2N}}\frac{ |v_k(y)|^{p_{s}^{\sharp}}
       |v_k(y)|^{p_{s}^{\sharp}}}{|x|^{\delta }|x-y|^{\mu}|y|^{\delta}} \dd x \dd y\\ 
& = \iint_{\mathbb{R}^{2N}}
\lambda _{ks} ^{\frac{N-sp-\theta}{p}2p^{\sharp}_s}\dfrac{|u_{k}(\lambda_{k} x)|^{p_{s}^{\sharp}}
       |u_k(\lambda_{k,1} y)|^{p_{s}^{\sharp}}}%
       {|x|^{\delta}
        |x-y|^{\mu}
        |y|^{\delta}}
\dd{x}\dd{y} \\
& \quad + \iint_{\mathbb{R}^{2N}}
\lambda _{k} ^{\frac{N-sp-\theta}{p}2p^{\sharp}_s}\dfrac{|v_{k}(\lambda_{k} x)|^{p_{s}^{\sharp}}
       |v_k(\lambda_{k} y)|^{p_{s}^{\sharp}}}%
       {|x|^{\delta}
        |x-y|^{\mu}
        |y|^{\delta}}
\dd{x}\dd{y} \\
& = 
\iint_{\mathbb{R}^{2N}}\frac{|u_k(\Bar{x})|^{p^\sharp_s}|u_k(\Bar{y})|^{p^\sharp_s}}{|\Bar{x}|^{\delta }|\Bar{x}-\Bar{y}|^{\mu}|\Bar{y}|^{\delta}} \dd \Bar{x} \dd \Bar{y}
+ \iint_{\mathbb{R}^{2N}}\frac{ |v_k(\Bar{x})|^{p_{s}^{\sharp}}
       |v_k(\Bar{y})|^{p_{s}^{\sharp}}}{|\Bar{x}|^{\delta }|\Bar{x}-\Bar{y}|^{\mu}|\Bar{y}|^{\delta}} \dd \Bar{x} \dd \Bar{y}\\ 
& =  Q^{\sharp}(u_k,v_k) = 1,
\end{align*}

Furthermore, $\|(\Tilde{u}_k,\Tilde{v}_k)\|^p\to S^{\sharp} $. In fact, we know that $\{(\Tilde{u}_k,\Tilde{v}_k)\}_{k \in \mathbb{N}}$ is a minimizing sequence for $S^{\sharp}$.
Using the same change of variables $\Bar{x}=\lambda _{k} x$ and $\Bar{y}=\lambda _{k} y$, we obtain
\begin{align*}
    \|(\Tilde{u}_k,\Tilde{v}_k)\|^p
    & = \iint_{\mathbb{R}^{2N}} \frac{|u_k(x)-u_k(y)|^p}{|x|^{\theta_1}|x-y|^{N+sp}|y|^{\theta_2}} \dd x \dd y + \iint_{\mathbb{R}^{2N}} \frac{|v_k(x)-v_k(y)|^p}{|x|^{\theta_1}|x-y|^{N+sp}|y|^{\theta_2}} \dd x \dd y \\
    & \quad- \gamma_1 \int_{\mathbb{R}^N} \frac{|u_k|^{p}}{|\Bar{x}|^{sp+\theta}}\dd x - \gamma_2 \int_{\mathbb{R}^N} \frac{|v_k|^{p}}{|\Bar{x}|^{sp+\theta}}\dd \Bar{x} \\
    & = \iint_{\mathbb{R}^{2N}} \lambda_{k,1}^{N-sp-\theta}\frac{|u_k(\lambda_{k}x)-u_k(\lambda_{k}y)|^p}{|x|^{\theta_1}|x-y|^{N+sp}|y|^{\theta_2}} \dd x \dd y \\ 
    & \quad +\iint_{\mathbb{R}^{2N}} \lambda_{k}^{N-sp-\theta}\frac{|v_k(\lambda_{k}x)-v_k(\lambda_{k}y)|^p}{|x|^{\theta_1}|x-y|^{N+sp}|y|^{\theta_2}} \dd x \dd y  \\
    & \quad- \gamma_1 \int_{\mathbb{R}^N} \lambda_{k}^{N-sp-\theta}\frac{|u_k|^{p}}{|x|^{sp+\theta}}\dd x - \gamma_2 \int_{\mathbb{R}^N}\lambda_{k,2}^{N-sp-\theta} \frac{|v_k|^{p}}{|x|^{sp+\theta}}\dd x \\ 
    & = \iint_{\mathbb{R}^{2N}} \frac{|u_k(\Bar{x})-u_k(\Bar{y})|^p}{|\Bar{x}|^{\theta_1}|\Bar{x}-\Bar{y}|^{N+sp}|\Bar{y}|^{\theta_2}} \dd \Bar{x} \dd \Bar{y} + \iint_{\mathbb{R}^{2N}} \frac{|v_k(\Bar{x})-v_k(\Bar{y})|^p}{|\Bar{x}|^{\theta_1}|\Bar{x}-\Bar{y}|^{N+sp}|\Bar{y}|^{\theta_2}} \dd \Bar{x} \dd \Bar{y} \\
    & \quad- \gamma_1 \int_{\mathbb{R}^N} \frac{|u_k|^{p}}{|\Bar{x}|^{sp+\theta}}\dd x - \gamma_2 \int_{\mathbb{R}^N} \frac{|v_k|^{p}}{|\Bar{x}|^{sp+\theta}}\dd \Bar{x} \\
       & = \|(u_k,v_k)\|^p.
\end{align*}
And since $\|(u_k,v_k)\|^p\to S^{\sharp}$ as $k\to+\infty$, we deduce that
$\|(\Tilde{u}_k,\Tilde{v}_k)\|^p\to S^{\sharp}$ as $k\to+\infty$.

In this way, the sequence
$\{(\Tilde{u}_k,\Tilde{v}_k)\}_{k\in\mathbb{N}} \subset W$ is also a minimizing sequence for $S^{\sharp}$ such that
we have
\begin{align}
    \label{4.1achineses:sistemas}
    Q^{\sharp}(\Tilde{u}_k,\Tilde{v}_k) = 1, \qquad \|(\Tilde{u}_k,\Tilde{v}_k)\|^p \to S^{\sharp}.
\end{align}

Consider $\{\Tilde{x}_k\}=\{\frac{x_k}{\lambda_{k}}\}$, from inequality~\eqref{des:4.1chineses:sistemas} together with Hölder's inequality 
\begin{align*}
    0< C &\leqslant \int_{B_1(\Tilde{x}_k)} \frac{|\Tilde{u}_k(x)|^p}{|x|^{pr}} \dd x \\
   & \leqslant \left(\int_{B_1(\Tilde{x}_k)} 1 \dd x\right)^{1-\frac{p}{p^*_s(\alpha,\theta)}} \left(\int_{B_1(\Tilde{x}_k)} \left(\frac{|\Tilde{u}_k(x)|^p}{|x|^{pr}}\right)^{\frac{p^*_s(\alpha,\theta)}{p}}\dd x \right)^{\frac{p}{p^*_s(\alpha,\theta)}} \\ 
    & \leqslant C \left(\int_{B_1(\Tilde{x}_k)} \frac{|\Tilde{u}_k(x)|^{p
    ^*_s(\alpha,\theta)}}{|x|^{\alpha}}\dd x\right)^{\frac{p}{p^*_s(\alpha,\theta)}}.
\end{align*}
Therefore, 
\begin{align}
    \label{des:4.2chineses:sistemas}
    \left(\int_{B_1(\Tilde{x}_k)} \frac{|\Tilde{u}_k(x)|^{p
    ^*_s(\alpha, \theta)}}{|x|^{\alpha}}\dd x\right)^{\frac{p}{p^*_s(\alpha, \theta)}} \geqslant C>0.
\end{align}

We claim that the sequence $\{\Tilde{x}_k\}\subset \mathbb{R}^{N}$ of the centers of the balls is bounded. We argue by contradiction and suppose that $|\Tilde{x}_k| \to + \infty$ as $k \to +\infty$; then for any $x \in B_1(\Tilde{x}_k)$,  we have $ |x| \geqslant |\Tilde{x}_k|-1$ for $k \in \mathbb{N}$ large enough. By Hölder's inequality, we obtain
\begin{align*}
    \int_{B_1(\Tilde{x}_k)} \frac{|\Tilde{u}_k(x)|^{p^*_s(\alpha, \theta)}}{|x|^{\alpha}} \dd x & \leqslant  \frac{1}{(|\Tilde{x}_k|-1)^{\alpha}}\int_{B_1(\Tilde{x}_k)}|\Tilde{u}_k(x)|^{p^*_s(\alpha, \theta)}\dd x \\
     \qquad & \leqslant \frac{C}{(|\Tilde{x}_k|-1)^{\alpha}}\left(\int_{B_1(\Tilde{x}_k)} |\Tilde{u}_k(x)|^{p^*_s(0,\theta)}\dd x\right)^{\frac{p^*_s(\alpha,\theta)}{p^*_s(0,\theta)}}\\
     \qquad & \leqslant  \frac{C}{(|\Tilde{x}_k|-1)^{\alpha}}\|\Tilde{u}_k(x)\|_{L^{p^*_s(0,\theta)}}^{p^*_s(\alpha,\theta)}\\
     & \leqslant  \frac{C}{(|\Tilde{x}_k|-1)^{\alpha}}\|u_k(x)\|_{\Dot{W}^{s,p}_{\theta}(\mathbb{R}^N)}^{p^*_s(\alpha, \theta)} \\
      & \leqslant \frac{C}{(|\Tilde{x}_k|-1)^{\alpha}} \to 0 \qquad(k\to+\infty)
\end{align*}
where we used the boundedness of the minimizing sequence $\{\Tilde{u}_k\}_{k\in\mathbb{N}}\subset \Dot{W}^{s,p}_{\theta}(\mathbb{R}^N) $. 
This is a contradiction with inequality~\eqref{des:4.2chineses:sistemas} and this implies that the sequence $\{\Tilde{x}_k\}\subset \mathbb{R}^{N}$ is bounded. 

From inequality~\eqref{des:4.1chineses:sistemas} and the boundedness of
the sequence $\{\Tilde{x}_k\}\subset \mathbb{R}^{N}$ of the centers of the balls, we may find $R>0$ such that $B_{R}(0)$ contains all balls of center $\Tilde{x}_k$ and radius $1$; moreover, with
\begin{align}
    \label{des:4.3chin}
    \int_{B_R(0)} \frac{|\Tilde{u}_k(x)|^p}{|x|^{pr}}\dd x \geqslant C_1>0.
\end{align}

Similar computations can also be done with respect to the sequence of functions $\{\tilde{v}_{k}\}_{k\in\mathbb{N}} \subset \Dot{W}^{s,p}_{\theta}(\mathbb{R}^N)$. Therefore, we obtain $\|(\Tilde{u}_k,\Tilde{v}_k)\|=\|(u_k,v_k)\|\leqslant C$ for $k \in \mathbb{N}$ large enough and there exists a  function pair $(\Tilde{u},\Tilde{v}) \in W$ such that 
\begin{align}
    \label{4.4chineses:sistemas}
    (\Tilde{u}_k,\Tilde{v}_k) \rightharpoonup (\Tilde{u},\Tilde{v})& \quad \textup{ weakly in $\Dot{W}^{s,p}_{\theta} (\mathbb{R}^N) \times \Dot{W}^{s,p}_{\theta} (\mathbb{R}^N)$}\\
    (\Tilde{u}_k,\Tilde{v}_k) \to (\Tilde{u},\Tilde{v}) & \quad \textup{a.e. on $\mathbb{R}^N \times \mathbb{R}^N$},
\end{align}
as $k \to +\infty$, up to subsequences. According to Lemma~\ref{lemma:2.3chineses}, we have 
\begin{align*}
     \Bigl(\frac{\Tilde{u}_k}{|x|^{r}},\frac{\Tilde{v}_k}{|x|^{r}}\Bigr) \to \Bigl(\frac{\Tilde{u}}{|x|^{r}},\frac{\Tilde{v}}{|x|^{r}}\Bigr) \; \textup{in $L^p_{\loc}(\mathbb{R}^N)\times L^p_{\loc}(\mathbb{R}^N)$};
\end{align*}
hence, 
\begin{align*}
    \int_{B_R(0)} \frac{|\Tilde{u}(x)|^p}{|x|^{pr}}\dd x \geqslant C_1>0,
\end{align*}
and we deduce that $\Tilde{u}\not\equiv 0$. Similarly, we can get $\Tilde{v}\not\equiv 0$.

We can verify in the same way as we did in~\cite[Lemma $2.6$, p.~$16$]{assunccao2023fractional} that
\begin{align}
\label{limiteqsus3}
    1 = Q^{\sharp}(\Tilde{u}_k,\Tilde{v}_k) = Q^{\sharp}(\Tilde{u}_k-\Tilde{u},\Tilde{v}_k-\Tilde{v})+ Q^{\sharp}(\Tilde{u},\Tilde{v})+ o(1).
\end{align}

By a Br\'{e}zis-Lieb's lemma, we obtain 
\begin{align}
\label{BL1}
    \begin{split}
    &\int_{\mathbb{R}^N} \Bigl(I_{\mu} \ast \frac{\left|\Tilde{u}_k - \Tilde{u}\right|}{|y|^{\delta}}^{p^{\sharp}_s}\Bigr) \frac{\left|\Tilde{u}_k - \Tilde{u}\right|}{|x|^{\delta}}^{p^{\sharp}_s}\dd x + \int_{\mathbb{R}^N} \Bigl(I_{\mu} \ast \frac{\left|\Tilde{u}\right|}{|y|^{\delta}}^{p^{\sharp}_s}\Bigr) \frac{\left|\Tilde{u}\right|}{|x|^{\delta}}^{p^{\sharp}_s} \dd x \\ 
     &\quad = \int_{\mathbb{R}^N} \Bigl(I_{\mu} \ast \frac{\left|\Tilde{u}_k\right|}{|y|^{\delta}}^{p^{\sharp}_s}\Bigr) \frac{\left|\Tilde{u}_k\right|}{|x|^{\delta}}^{p^{\sharp}_s} \dd x + o(1),
     \end{split}
\end{align}
and
\begin{align}
\label{BL2}
    \begin{split}
    & \int_{\mathbb{R}^N} \Bigl(I_{\mu} \ast \frac{\left|\Tilde{v}_k - \Tilde{v}\right|}{|y|^{\delta}}^{p^{\sharp}_s}\Bigr) \frac{\left|\Tilde{v}_k - \Tilde{v}\right|}{|x|^{\delta}}^{p^{\sharp}_s}\dd x + \int_{\mathbb{R}^N} \Bigl(I_{\mu} \ast \frac{\left|\Tilde{v}\right|}{|y|^{\delta}}^{p^{\sharp}_s}\Bigr) \frac{\left|\Tilde{v}\right|}{|x|^{\delta}}^{p^{\sharp}_s} \dd x \\
    & \quad = \int_{\mathbb{R}^N} \Bigl(I_{\mu} \ast \frac{\left|\Tilde{v}_k\right|}{|y|^{\delta}}^{p^{\sharp}_s}\Bigr) \frac{\left|\Tilde{v}_k\right|}{|x|^{\delta}}^{p^{\sharp}_s} \dd x + o(1),
    \end{split}
\end{align}

Therefore, by definition~\eqref{Ssharp:sistemas}, by weak convergence $(\Tilde{u}_k,\Tilde{v}_k) \rightharpoonup (\Tilde{u},\Tilde{v})$ in $W$ together with the Brézis–Lieb lemma and by
the estimate~\eqref{limiteqsus3}, we have
\begin{align*}
    S^{\sharp} & = \lim\limits_{k \to \infty}\|(\Tilde{u}_k,\Tilde{v}_k)\|^p \\
    & = \|(\Tilde{u},\Tilde{v})\|^p +\lim\limits_{k \to \infty}\|(\Tilde{u}_k-\Tilde{u},\Tilde{v}_k-\Tilde{v})\|^p\\
    & \geqslant  S^{\sharp} (Q^{\sharp}(\Tilde{u},\Tilde{v}))^{\frac{p}{p^{\sharp}_{s}}}+  S^{\sharp}\left(\lim\limits_{k \to \infty}Q^{\sharp}(\Tilde{u}_k-\Tilde{u},\Tilde{v}_k-\Tilde{v})\right)^{\frac{p}{p^{\sharp}_{s}}}\\
    & \geqslant S^{\sharp} \bigr(Q^{\sharp}(\Tilde{u},\Tilde{v})+ \lim\limits_{k \to \infty}Q^{\sharp}(\Tilde{u}_k-\Tilde{u},\Tilde{v}_k-\Tilde{v})\bigl)^{\frac{p}{p^{\sharp}}}\\
    & =  S^{\sharp},
\end{align*}
where in the last but one passage above we used the inequality 
\begin{align}
\label{desigualdad}
    (a+b)^{q} \leqslant a^{q}+b^{q},
\end{align}
valid for all $a, b  \in \mathbb{R}_{+}^{*}$ and $0<q<1$. So we have equality in all passages, that is, 
\begin{align}
    \label{4.4achin}
    Q^{\sharp}(\Tilde{u},\Tilde{v})=1, \qquad \lim\limits_{k\to \infty}Q^{\sharp}(\Tilde{u}_k-\Tilde{u},\Tilde{v}_k-\Tilde{v})=0,
\end{align}
since $\Tilde{u}, \Tilde{v} \not\equiv 0$. It turns out that, since
\begin{align*}
    S^{\sharp} = \|(\Tilde{u},\Tilde{v})\|^p+\lim\limits_{k \to \infty}\|(\Tilde{u}_k-\Tilde{u},\Tilde{v}_k-\Tilde{v})\|^p,
\end{align*}
then
\begin{align*}
    S^{\sharp} =\|(\Tilde{u},\Tilde{v})\|^p \qquad \textup{and} \qquad \lim\limits_{k \to \infty}\|(\Tilde{u}_k-\Tilde{u},\Tilde{v}_k-\Tilde{v})\|^p=0.
\end{align*}

Finally, by inequality
\begin{align*}
&\iint_{\mathbb{R}^{2N}} 
\dfrac{||\Tilde{u}(x)|-|\Tilde{u}(y)||^{p}}{|x|^{\theta _1}|x-y|^{N+sp}|y|^{\theta _2}}
\dd x \dd y +  \iint_{\mathbb{R}^{2N}} 
\dfrac{||\Tilde{v}(x)|-|\Tilde{v}(y)||^{p}}{|x|^{\theta _1}|x-y|^{N+sp}|y|^{\theta _2}}
\dd x \dd y \\
& \qquad \leqslant \iint_{\mathbb{R}^{2N}} 
\dfrac{|\Tilde{u}(x)-\Tilde{u}(y)|^{p}}{|x|^{\theta _1}|x-y|^{N+sp}|y|^{\theta _2}}
\dd x \dd y + \iint_{\mathbb{R}^{2N}} 
\dfrac{|\Tilde{v}(x)-\Tilde{v}(y)|^{p}}{|x|^{\theta _1}|x-y|^{N+sp}|y|^{\theta _2}}
\dd x \dd y.
\end{align*} 
we deduce that $(|\Tilde{u}|, |\Tilde{v}|) \in W$ is also a minimizer for $S^{\sharp}$; so we can assume that $\Tilde{u}\geqslant 0, \Tilde{v}\geqslant 0$. Thus, $S^{\sharp}$ is achieved by a non-negative function in the case $0<\alpha<sp+\theta$ and $\gamma < \gamma _H$.

\bigskip

\item For $0< \beta <sp+\theta<N$ and $\gamma < \gamma _H$, let $\{(u_k,v_k)\}_{k\in\mathbb{N}}\subset W$ be a minimizing sequence for $S^*$ such that
        \begin{align}
            Q^{*}(u_k, v_k) = 1, \qquad \|(u_k,v_k)\|^p \to S^*
        \end{align}
as $k \to +\infty$.
         
Now we claim that $S^*$ is invariant under the previously defined dilation.
Let 
\begin{align*}
\Tilde{u}_k (x) 
&=\lambda _{k,1} ^{\frac{N-sp-\theta}{p}}u_k(\lambda _{k,1} x),
&\qquad
\Tilde{v}_k (x) 
&=\lambda _{k,2} ^{\frac{N-sp-\theta}{p}}v_k(\lambda _{k,2} x)
\end{align*}
and $\Tilde{x}_k =\frac{x_k}{\lambda _{k,i}}$ for $i \in \{1,2\}$ as in the previous case. 
In this way, the sequence $\{(\Tilde{u}_k, \Tilde{v}_k)\}_{k\in\mathbb{N}} \subset W$ is also a minimizing sequence for $S^*$ such that
we have
\begin{align*}
      Q^{*}(\Tilde{u}_k, \Tilde{v}_k) = 1, \qquad \|(\Tilde{u}_k,\Tilde{v}_k)\|^p \to S^*
\end{align*}

We have already shown that
$\| (\Tilde{u}_{k}, \Tilde{v}_k) \| = \| (u_{k}, v_k) \|$ for every $k \in \mathbb{N}$.
Hence,  $\|(\Tilde{u}_k, \Tilde{v}_k)\|^p\to S^*$.

We claim that the sequence $\{\Tilde{x}_k\} \subset \mathbb{R}^{N}$ is bounded and the proof follows the same steps already presented. 
 From this boundedness and inequality~\eqref{des:4.1chineses:sistemas}, we may find $R>0$ such that $B_{R}(0)$ contains all the unitary balls $B_{1}(\Tilde{x}_{k})$ centered in $\Tilde{x}_k$ and
\begin{align}
    \label{des:4.3achin}
    \int_{B_R(0)} \frac{|v_k(x)|^p}{|x|^{pr}}\dd x \geqslant C_1>0.
\end{align}
Since $\|v_k\|=\|u_k\|\leqslant C$, there exists a $v \in \Dot{W}^{s,p}_{\theta}(\mathbb{R}^N)$ such that 
\begin{align}
    \label{4.4bchin}
    v_k \rightharpoonup v\quad \textup{in $\Dot{W}^{s,p}_{\theta}(\mathbb{R}^N)$},\qquad v_k \to v \; \textup{a.e. \quad on $\mathbb{R}^N$},
\end{align}
as $k \to +\infty$, up to subsequences. According to Lemma~\ref{lemma:2.3chineses}, we have 
\begin{align*}
    \frac{v_k}{|x|^r} \to \frac{v}{|x|^r} \qquad \textup{in $L^p_{\loc}(\mathbb{R}^N)$},
\end{align*}
as $k \to +\infty$, where $r=\frac{\beta}{p^*_s}$. Therefore, 
\begin{align*}
    \int_{B_R(0)} \frac{|v(x)|^p}{|x|^{pr}}\dd x \geqslant C_1>0,
\end{align*}
and we deduce that $v\not\equiv 0$. 

We may verify by a variant of Brézis–Lieb Lemma~\cite[Lemma $2.2$, p. $13$]{assunccao2023fractional} that, if $q=p^*_s(\beta,\theta)$ and $\delta = \beta$, then
\begin{align*}
    1 = \int_{\mathbb{R}^N} \frac{|v_k|^{p^*_s}}{|x|^{\beta}}\dd x = \int_{\mathbb{R}^N} \frac{|v_k-v|^{p^*_s}}{|x|^{\beta}}\dd x +\int_{\mathbb{R}^N} \frac{|v|^{p^*_s}}{|x|^{\beta}}\dd x+ o(1).
\end{align*}
By definition~\eqref{Sast:sistemas} and by weak convergence $(\Tilde{u}_k,\Tilde{v}_k) \rightharpoonup (\Tilde{u},\Tilde{v})$ in $W$, we deduce that
\begin{align*}
    S^* & = \lim\limits_{k \to \infty}\|(\Tilde{u}_k,\Tilde{v}_k)\|^p \\
    & = \|(\Tilde{u},\Tilde{v})\|^p +\lim\limits_{k \to \infty}\|(\Tilde{u}_k-\Tilde{u},\Tilde{v}_k-\Tilde{v})\|^p\\
    & \geqslant  S^*(Q^*(\Tilde{u},\Tilde{v}))^{\frac{p}{p^*_s}} +  S^*\left(\lim\limits_{k \to \infty}Q^*(\Tilde{u}_k-\Tilde{u},\Tilde{v}_k-\Tilde{v})\right)^{\frac{p}{p^*_s}}\\
    & \geqslant S^* \bigr(Q^*(\Tilde{u},\Tilde{v})+ \lim\limits_{k \to \infty}Q^*(\Tilde{u}_k-\Tilde{u},\Tilde{v}_k-\Tilde{v})\bigl)^{\frac{p}{p^*_s}}\\
    & =  S^*
\end{align*}
where we used the inequality~\eqref{desigualdad}.
So we have equality in all
passages, that is, 
\begin{align}
    \label{4.4achineses1}
    Q^*(\Tilde{u},\Tilde{v})=1, \qquad \lim\limits_{k\to \infty}Q^*(\Tilde{u}_k-\Tilde{u},\Tilde{v}_k-\Tilde{v})=0,
\end{align}
since $\Tilde{u}, \Tilde{v} \not\equiv 0$. It turns out that, since
$    S^* = \|(\Tilde{u},\Tilde{v})\|^p+\lim\limits_{k \to \infty}\|(\Tilde{u}_k-\Tilde{u},\Tilde{v}_k-\Tilde{v})\|^p$,
then
$    S^* =\|(\Tilde{u},\Tilde{v})\|^p$ and $\lim\limits_{k \to \infty}\|(\Tilde{u}_k-\Tilde{u},\Tilde{v}_k-\Tilde{v})\|^p=0.$
 
         As in the previous case, we deduce that $|(\Tilde{u},\Tilde{v})| \in W$ is also a minimizer for $S^*$ is achieved by a non-negative function in the case $0<\beta<sp+\theta$ and $\gamma < \gamma _H$.

\end{enumerate}

\end{proof}

\section{Existence of Palais-Smale sequence}

We shall now use the minimizers of $S^{\sharp}$ and $S^*$ obtained in Proposition~\ref{minimizadores:sistemas} to prove the existence of a nontrivial weak solution for equation~\eqref{prob:sistemas}. Recall the definition~\eqref{funcionaldeenergia} of the energy functional associated to roblem~\eqref{prob:sistemas}. The fractional Sobolev and fractional Hardy-Sobolev inequalities imply that $I \in C^1(W, \mathbb{R})$ and that
\begin{align*}
    \lefteqn{\langle I'(u,v),(\phi_1,\phi_2)\rangle} \\
    & =\iint_{\mathbb{R}^{2N}}  \frac{|u(x)-u(y)|^{p-2}(u(x)-u(y))(\phi_1 (x)-\phi_1 (y))}{|x|^{\theta _1}|x-y|^{N+sp}|y|^{\theta _2}}\dd x \dd y \\
    & \quad +\iint_{\mathbb{R}^{2N}}  \frac{|v(x)-v(y)|^{p-2}(v(x)-v(y))(\phi_2 (x)-\phi_2 (y))}{|x|^{\theta _1}|x-y|^{N+sp}|y|^{\theta _2}}\dd x \dd y\\
   & \quad - \gamma_1 \int_{\mathbb{R}^N} \frac{|u|^{p-2}u\phi_1}{|x|^{sp+\theta}} \dd x - \gamma_2 \int_{\mathbb{R}^N} \frac{|v|^{p-2}v\phi_2}{|x|^{sp+\theta}} \dd x
   \\ & \quad -\iint_{\mathbb{R}^{2N}}\frac{|u(x)|^{p^\sharp_s-2}|u(y)|^{p^\sharp_s}u(x)\phi_1(x)}{|x|^{\delta }|x-y|^{\mu}|y|^{\delta }}\dd x\dd y -\iint_{\mathbb{R}^{2N}}\frac{|v(x)|^{p^\sharp_s-2}|v(y)|^{p^\sharp_s}v(x)\phi_2(x)}{|x|^{\delta }|x-y|^{\mu}|y|^{\delta }}\dd x\dd y \\
   & \quad -\int_{\mathbb{R}^N} \frac{|u|^{p^*_s-2}u(x)\phi_1 }{|x|^{\beta}}\dd x -\int_{\mathbb{R}^N} \frac{|v|^{p^*_s{}-2}v(x)\phi_2 }{|x|^{\beta}}\dd x\\
    & \quad + \int_{\mathbb{R}^N}\frac{\eta a |u|^{a-2}u\phi_1|v|^b}{|x|^{\beta}}\dd x + \int_{\mathbb{R}^N}\frac{\eta b |u|^{a}|v|^{b-2}v\phi_2}{|x|^{\beta}}\dd x.
\end{align*}
Note that a nontrivial critical point of $I$ is a nontrivial weak solution to equation~\eqref{prob:sistemas}.

Recall that a Palais-Smale sequence for the energy functional $I$ at the level $c \in \mathbb{R}$ is a sequence 
$\{(u_{k},v_k)\}_{k\in\mathbb{N}} \subset W$ such that 
    \begin{align}
        \label{lim:5.2chin}
        \lim\limits_{k \to +\infty} I(u_k,v_k) = c \quad \textup{and} \quad \lim\limits_{k \to +\infty} I'(u_k,v_k)=0 \quad \textup{strongly in $W'$}.
    \end{align}
This sequence is referred to as a $(PS)_{c}$ sequence.

Now we state a result that ensures the existence of a Palais-Smale sequence for the energy functional.
\begin{proposition}
    \label{prop:5.2chineses:sistemas}
    Let $ s \in (0,1), 0<\alpha,\beta<sp+\theta<N, \mu \in(0,N)$ and $\gamma < \gamma _H$. Consider the functional $I \colon W \to \mathbb{R}$ defined in~\eqref{funcionaldeenergia} on the Banach space $W$. Then there exists a $(PS)_{c}$ sequence $\{(u_k,v_k)\}\subset W$  for $I$ at some level $c \in (0, c^*)$, 
where
\begin{align}
    \label{def:cestrela:sistemas}
    c^* 
    & \coloneqq \min \biggl\{
\Bigl(\frac{1}{p}-\frac{1}{2p^{\sharp}_{s}(\delta,\theta,\mu)}\Bigr)
    S^{\sharp\frac{2p^{\sharp}_{s}}{2p^{\sharp}_{s}(\delta,\theta,\mu)-p}}, \Bigl(\frac{1}{p}-\frac{1}{p^{\ast}_{s}(\beta,\theta)}\Bigr)
    S^{*\frac{p^{\ast}_{s}(\beta,\theta)}{p^{\ast}_{s}(\beta,\theta)-p}}\biggr\}.
\end{align}

\end{proposition}

To prove Proposition~\ref{prop:5.2chineses:sistemas} we need the following version of the mountain pass theorem by Ambrosetti \& Rabinowitz~\cite{ambrosetti}. 

\begin{lemma}(Montain Pass Lemma) 
\label{passodamontanha:sistemas}
Let $(W, \| \cdot \|)$ be a Banach space and let $I \in C^1(W,\mathbb{R})$ a functional such that the following conditions are satisfied:
\begin{itemize}[wide]
    \item [$(1)$] $I(0,0)=0$;
    \item [$(2)$] There exist $\rho >0$ and $ r>0$ such that $I(u,v)\geqslant \rho$ for all $u,v \in \Dot{W}^{s,p}_{\theta}(\mathbb{R}^N)$ with $\|(u,v)\|=r;$
    \item [$(3)$] There exist $e \in W$ with $\|e\|>r$ such that $I(e)<0$; define
    \begin{align*}
        c \coloneqq \inf\limits_{g \in \Gamma} \sup\limits_{t \in [0,1]}I(g(t)),
    \end{align*}
    where
    \begin{align*}
        \Gamma \coloneqq \left\{g \in C^0([0,1],\Dot{W}^{s,p}_{\theta}(\mathbb{R}^N)) \colon g(0)=0, g(
        e)<0\right\}.
    \end{align*}
    \end{itemize}
    Then $c\geqslant \rho >0$, and there exists a $(PS)_{c}$ sequence $\{(u_k,v_k)\}\subset W$ for $I$ at level $c $, i.e.,
    \begin{align*}
        \lim\limits_{k \to +\infty} I(u_k,v_k) = c \quad \textup{and} \quad \lim\limits_{k \to +\infty} I'(u_k,v_k)=0 \quad \textup{strongly in $W'$}.
    \end{align*}
\end{lemma}

The proof of Proposition~\ref{prop:5.2chineses:sistemas} follows from the next two lemmas.

\begin{lemma}
    \label{lema1passodamontanha:sistemas}
    The functional $I$ verifies the assumptions of Lemma~\ref{passodamontanha:sistemas}.
\end{lemma}

\begin{proof}
     Clearly, we have $I(0,0)=0$. 
     We now verify the second assumption of Lemma~\ref{passodamontanha:sistemas}.  Recalling the definition~\eqref{qestrela:sistemas} of the quadratic form $Q^{*}$ and using  inequality~\eqref{limitacaopsharp}, for any $ (u,v) \in W$ we obtain
    \begin{align*}
        I(u,v)  
        & \geqslant  \frac{1}{p}\|(u,v)\|_W^p 
        - \frac{C}{2p^\sharp_s} \bigl[\|u\|^{2p^\sharp_s}_{\Dot{W}^{s,p}_{\theta}(\mathbb{R}^N)}
         + \|v\|^{2p^\sharp_s}_{\Dot{W}^{s,p}_{\theta}(\mathbb{R}^N)}\bigr]
        -\frac{1}{p^{*}_{s}}Q^{*}(u,v) \\
         & \geqslant  \frac{1}{p}\|(u,v)\|_W^p 
        - \frac{C}{2p^\sharp_s} \bigl[\|u\|^p_{\Dot{W}^{s,p}_{\theta}(\mathbb{R}^N)}
         + \|v\|^p_{\Dot{W}^{s,p}_{\theta}(\mathbb{R}^N)}\bigr]^{\frac{2p^\sharp_s}{p}}
        -\frac{1}{p^{*}_{s}}Q^{*}(u,v) \\ 
         & \geqslant  \frac{1}{p}\|(u,v)\|_W^p 
        - C_1 \|(u,v)\|^{2p^\sharp_s}
       - C_2\|(u,v)\|^{p^*_s}.
    \end{align*}

    Since $s \in (0,1), 0 < \alpha, \beta <sp +\theta <N$ and $\mu \in (0,N)$, we have that $p^*_s(\beta,\theta)>p$ and $2p^\sharp_s>p^*_s(\alpha,\theta)>p$. Therefore, there exists $r>0$ small enough such that
    \begin{align*}
        \inf\limits_{\|(u,v)\|=r}I(u,v)>\rho,
    \end{align*}
    so item $(2)$ of Lemma~\ref{passodamontanha:sistemas} are satisfied.

For $(u,v) \in W$ and $t \in \mathbb{R}_{+}$, we have 
\begin{align*}
    I(tu,tv)=\frac{t^p}{p}\|(u,v)\|^p- \frac{t^{2p^{\sharp}_{s}}}{2p^{\sharp}_{s}}Q^{\sharp}(u,v)-\frac{t^{p^*_s}}{p^*_s} Q^*(u,v) ;
\end{align*}
since $2p^\sharp_s>p^*_s(\alpha,\theta)>p$, we deduce that
\begin{align*}
    \lim\limits_{t \to + \infty} I(tu,tv)=-\infty \quad \textup{for any $(u,v) \in W$}.
\end{align*}
Consequently, for any fixed $e \in W$, there exists $t_e>0$ such that $\|t_ee\|>r$ and $I(t_ee)<0$. Thus, item $(3)$ of Lemma~\ref{passodamontanha:sistemas} is satisfied.
\end{proof}

From Lemma~\ref{lema1passodamontanha:sistemas} above, we guarante by Lemma~\ref{passodamontanha:sistemas} the existence of a Palais-Smale sequence $\{(u_k,v_k)\} \subset W$ such that
\begin{align*}
        \lim\limits_{k \to +\infty} I(u_k,v_k) = c \quad \textup{and} \quad \lim\limits_{k \to +\infty} I'(u_k,v_k)=0 \quad \textup{strongly in $W'$}.
    \end{align*}
Moreover, by the definition of $c$ we deduce that $c \geqslant \rho > 0 $. Therefore $c>0$ for all function $(u,v) \in W\setminus\{(0,0)\}$.

\begin{lemma}
  \label{lema2passodamontanha:sistemas}
  Suppose that $\mu \in (0,N)$ and that $0< \alpha< sp+ \theta$. Then there exists $(u,v) \in W \setminus \{(0,0)\}$ such that $c \in (0, c^*)$, where $c^*$ is defined in~\eqref{def:cestrela:sistemas}.
\end{lemma}

\begin{proof}

Using Proposition~\ref{minimizadores:sistemas}, we obtain the minimizers $(u_1,v_1) \in W$ for $S^{\sharp}$ and $(u_2,v_2) \in W$ for $S^* $, respectively. Thus, there exist a function $(e_1,e_2)\in W$ defined by
\begin{align*}
    (e_1,e_2) = 
    \begin{cases}
        (u_1,v_1), \quad \textup{if} \; \frac{2p^\sharp_s-p}{2pp^\sharp_s}S_{\mu}(N,s,\gamma,\alpha)^{\frac{2p^{\sharp}_{\mu}(\alpha, \theta)}{2p^{\sharp}_{\mu}(\alpha, \theta)-p}}\leqslant \frac{sp-\beta}{p(N-\beta)}\Lambda (N, s, \gamma,\beta)^{\frac{N-\beta}{sp-\beta}} \\
       (u_2,v_2), \quad \textup{if} \; \dfrac{2p^\sharp_s(\delta,\theta,\mu)-p}{2pp^\sharp_s}S_{\mu}(N,s,\gamma,\alpha)^{\frac{2p^{\sharp}_{\mu}(\alpha, \theta)}{2p^{\sharp}_{\mu}(\alpha, \theta)-p}} > \frac{sp-\beta}{p(N-\beta)}\Lambda (N, s, \gamma,\beta)^{\frac{N-\beta}{sp-\beta}}
    \end{cases}
\end{align*}
such that $\|(e_1,e_2)\|>r$ and $I(e_1,e_2)<0$. We can define
\begin{align*}
    c \coloneqq \inf\limits_{g \in \Gamma}\sup\limits_{t \in [0,1]} I(g(t)),
\end{align*}
where 
\begin{align*}
    \Gamma \coloneqq \left\{g \in C^0([0,1], \Dot{W}^{s,p}_{\theta}(\mathbb{R}^N)) \colon g(0)=0,g(e_1,e_2)<0\right\}.
\end{align*}
Clearly, we have that $c>0$. For the case where $e=(u_1,v_1)$, we can deduce that 
\begin{align*}
    0<c< \frac{2p^\sharp_s-p}{2pp^\sharp_s}S_{\mu}(N,s,\gamma,\alpha)^{\frac{p^\sharp_s}{p^*_{\mu}(\alpha,\theta)-1}}. 
\end{align*}
In fact, for all $ t \geqslant 0$, by the definition of the functional $I$, we have that
\begin{align*}
    I(tu_1,tv_1) \leqslant \frac{t^p}{p}\|(u_1,v_1)\|^p - \frac{t^{2p^\sharp_s}}{2p^\sharp_s} Q^{\sharp}(u_1,v_1) \eqqcolon f_1(t).
\end{align*}
It is easy to see that 
$    f_1'(t) = t^{p-1}[\|(u_1,v_1)\|^p - t^{2p^\sharp_s -p}Q^{\sharp}(u_1,v_1)]$. So, $f'_{1}(\Tilde{t})=0$ for
\begin{align}
\label{maxf1}
    \Tilde{t}=\left(\frac{\|(u_1,v_1)\|^p}{Q^{\sharp}(u_1,v_1)}\right)^{\frac{1}{2p^\sharp_s-p}},
\end{align}
and this is a point of maximum for $f_{1}$. Besides of that, this maximum value is
\begin{align*}
    f_1(\Tilde{t}) 
     & = \left[\frac{1}{p}-\frac{1}{2p^\sharp_s}\right] \frac{\|(u_1,v_1)\|^{\frac{2pp^\sharp_s}{2p^\sharp_s-p}}}{Q^{\sharp}(u_1,v_1)^{\frac{p}{2p^\sharp_s-p}}} 
    = \left[\frac{2p^\sharp_s-p}{2pp^\sharp_s}\right]S^{\sharp \frac{2p^\sharp_s}{2p^\sharp_s-p}}.
\end{align*}
Therefore,
\begin{align}
\label{des:5.3chineses:sistemas}
    \sup\limits_{t \geqslant 0}I(tu_1,tv_1) \leqslant \sup\limits_{t \geqslant 0} f_1(t)  
    = \frac{2p^\sharp_s-p}{2pp^\sharp_s(\delta,\theta,\mu)} S_{\sharp}^{\frac{2p^\sharp_s}{2p^\sharp_s-p}}
\end{align}
The equality does not hold in~\eqref{des:5.3chineses:sistemas}; otherwise, we would have that $ \sup\limits_{t \geqslant 0}I(tu_1,tv_1) = \sup\limits_{t \geqslant 0} f_1(t)$. Let $t_1>0$ be the point where $\sup\limits_{t \geqslant 0}I(tu_1,tv_1)$ is attained. We have
\begin{align*}
    f_1(t_1) -\frac{t_1^{p^*_s(\beta,\theta)}}{p^*_s(\beta,\theta)}Q^*(u_1,v_1) = f_1(\Tilde{t})
\end{align*}
which means that $f_1(t_1)>f_1(\Tilde{t})$, since $t_1>0$. This contradicts the fact that $\Tilde{t}$ is the unique maximum point for $f_1$. Thus, we have strict inequality in~\eqref{des:5.3chineses:sistemas}, that is,
\begin{align}
    \label{des:5.4chin}
     \sup\limits_{t \geqslant 0}I(tu_1,tv_1) < \sup\limits_{t \geqslant 0} f_1(t)  = \frac{2p^{\sharp}_{\mu}(\alpha, \theta)-p}{2pp^{\sharp}_{\mu}(\alpha, \theta)} S^{\sharp \frac{2p^{\sharp}_{\mu}(\alpha, \theta)}{2p^{\sharp}_{\mu}(\alpha, \theta)-p}}.
\end{align}
Therefore, $0<c<\frac{2p^{\sharp}_{\mu}(\alpha, \theta)-p}{2pp^{\sharp}_{\mu}(\alpha, \theta)} S^{\sharp \frac{2p^{\sharp}_{\mu}(\alpha, \theta)}{2p^{\sharp}_{\mu}(\alpha, \theta)-p}}$.

Similarly, for the case of $e=(u_2,v_2)$, we can verify that
\begin{align}
    \label{des:5.5chineses:sistemas}
    \sup\limits_{t\geqslant 0} I(tu_2,tv_2) < \frac{sp-\beta}{p(N-\beta)}\Lambda (N,s,\gamma,\beta)^{\frac{N-\beta}{sp-\beta}}.
\end{align}
In fact, for all $ t \geqslant 0$, by functional $I$ definition we have that
\begin{align*}
    I(tu_2,tv_2) \leqslant \frac{t^p}{p}\|(u_2,v_2)\|^p - \frac{t^{p^*_s}}{p^*_s} Q^*(u,v)\coloneqq g_1(t).
\end{align*}
It is easy to see that
$    g_1'(t) = t^{p-1}\left[\|(u_2,v_2)\|^p - t^{p^*_s -p}Q^*(u,v) \right]$. So, $g_1(\Tilde{t}) =0$ for 
\begin{align*}
\Tilde{t}=\left(\frac{\|(u_2,v_2)\|^p}{Q^*(u,v)}\right)^{\frac{1}{p^*_s-p}},
\end{align*}
and this is a point of maximum for $g_1$. Besides of that, this maximum value is
\begin{align*}
    g_1(\Tilde{t}) 
     & =\left[\frac{1}{p}-\frac{1}{p^*_s}\right] \left(\frac{\|(u_2,v_2)\|^p}{Q^{*\frac{p}{p^*_s}}}\right)^{\frac{p^*_s}{p^*_s-p}}
     = \frac{sp+\theta-\beta}{p(N-\beta)} S^{*\frac{N-\beta}{sp+\theta-\beta}}.
\end{align*}
Therefore,
\begin{align}
\label{des:5.3achineses:sistemas}
    \sup\limits_{t \geqslant 0}I(tu_2,tv_2) &\leqslant \sup\limits_{t \geqslant 0} g_1(t) = \frac{sp+\theta-\beta}{p(N-\beta)} S^{*\frac{N-\beta}{sp+\theta-\beta}}.
\end{align}
The equality does not hold in~\eqref{des:5.3achineses:sistemas}, otherwise, we would have that $ \sup\limits_{t \geqslant 0}I(tu_2,tv_2) = \sup\limits_{t \geqslant 0} g_1(t)$. Let $t_1>0$, where $\sup\limits_{t \geqslant 0}I(tu_2,tv_2)$ is attained. We have
\begin{align*}
    g_1(t_1) -\frac{t_1^{2p^\sharp_s}}{2p^\sharp_s}Q^{\sharp}(u_2,v_2) = g_1(\Tilde{t})
\end{align*}
which means that $g_1(t_1)>g_1(\Tilde{t})$, since $t_1>0$. This contradicts the fact that $\Tilde{t}$ is the unique maximum point for $g_1(t)$. Thus
\begin{align}
    \label{des:5.4achineses:sistemas}
     \sup\limits_{t \geqslant 0}I(tu_2,tv_2) &< \sup\limits_{t \geqslant 0} g_1(t) = \frac{sp+\theta
     -\beta}{p(N-\beta)}S^{*\frac{N-\beta}{sp+\theta-\beta}}.
\end{align}
Therefore, $0<c<\frac{sp+\theta-\beta}{p(N-\beta)}S^{^*\frac{N-\beta}{sp+\theta-\beta}}$.

From the definition~\eqref{def:cestrela:sistemas} of $c^{\ast}$ and from inequalities~\eqref{des:5.4chin} and~\eqref{des:5.4achineses:sistemas}, we have
\begin{align*}
    0<c < c^* \coloneqq \min \biggl\{
\Bigl(\frac{1}{p}-\frac{1}{2p^{\sharp}_{s}}\Bigr)
    S^{\sharp\frac{2p^{\sharp}_{s}}{2p^{\sharp}_{s}-p}}, \Bigl(\frac{1}{p}-\frac{1}{p^{\ast}_{s}}\Bigr)
    S^{*\frac{p^{\ast}_{s}}{p^{\ast}_{s}-p}}\biggr\}.
\end{align*}
The lemma is proved.
\end{proof}

\begin{proof}[Proof of Proposition~\ref{prop:5.2chineses:sistemas}]
Follows immediately from Lemmas~\ref{lema1passodamontanha:sistemas} and~\ref{lema2passodamontanha:sistemas}.
\end{proof}

\section{Proof of Theorem~\ref{teo:sistemas}}
The existence of a solution will follow from the proof of the Theorem~\ref{teo:sistemas}.

\begin{proof}[Proof of Theorem~\ref{teo:sistemas}]

Suppose that $s \in (0,1), 0<\alpha, \beta< sp+\theta, \mu \in (0,N)$ and $\gamma < \gamma _H$.

    Let $\{(u_k,v_k)\}_{k \in \mathbb{N}}\subset W$ be a Palais-Smale sequence $(PS)_{c}$ as in Proposition~\ref{prop:5.2chineses:sistemas}, i.e.,
    \begin{align*}
        I(u_k,v_k) \to c, \; I'(u_k,v_k) \to 0 \qquad \textup{strongly in $W'$ as $k \to + \infty$}.
    \end{align*}
    Then 
    \begin{align}
        \label{eq:5.6chineses:sistemas}
        I(u_k,v_k)= \frac{1}{p}\|(u_k,v_k)\|^p - \frac{1}{2p^\sharp_s}Q^{\sharp}(u_k,u_k) -\frac{1}{p^*_s}Q^*(u_k,v_k) =c+o(1)
    \end{align}
    and 
    \begin{align}
        \label{eq:5.7chineses:sistemas}
         \langle I'(u_k,v_k),(u_k,v_k) \rangle =\|(u_k,v_k)\|^p -Q^{\sharp}(u_k,v_k)-Q^*(u_k,v_k) =o(1).
    \end{align}

    From~\eqref{eq:5.6chineses:sistemas} and~\eqref{eq:5.7chineses:sistemas}, if $2p^\sharp_s\geqslant p^*_s>p$, we have
\begin{align*}
    c+o(1)\|(u_k,v_k)\| & = I(u_k,v_k)-\frac{1}{p^*_s}\langle I'(u_k,v_k),(u_k,v_k)\rangle \\
    &=\frac{p^*_s -p}{p\cdot p^*_s} \|(u_k,v_k)\|^p + \left(\frac{1}{p^*_s}-\frac{1}{2p^\sharp_s}\right)Q^{\sharp}(u_k,v_k)\\
    & \geqslant  \frac{p^*_s -p}{p\cdot p^*_s} \|(u_k,v_k)\|^p.
\end{align*}

Again from~\eqref{eq:5.6chineses:sistemas} and~\eqref{eq:5.7chineses:sistemas}, if $ p^*_s> 2p^\sharp_s>p$, we have
\begin{align*}
    c+o(1)\|(u_k,v_k)\| & = I(u_k,v_k)-\frac{1}{2p^\sharp_s
 }\langle I'(u_k,v_k),(u_k,v_k)\rangle \\
    &=\frac{2p^\sharp_s
  -p}{p\cdot 2p^\sharp_s
 } \|(u_k,v_k)\|^p + \left(\frac{1}{2p^\sharp_s
 }-\frac{1}{p^*_s}\right)Q^*(u_k,v_k)\\
    & \geqslant \frac{2p^\sharp_s
  -p}{p\cdot 2p^\sharp_s
 } \|(u_k,v_k)\|^p.
\end{align*}
Thus, $\{(u_k,v_k)\}_{k \in \mathbb{N}} \subset W$ is a bounded sequence; so from the estimate~\eqref{eq:5.7chineses:sistemas} there exists a subsequence, still denoted by $\{(u_k,v_k)\}_{k\in\mathbb{N}}\subset W$, such that 
$\|(u_k,v_k)\|^p \to b, 
Q^*(u_k,v_k) \to d_1,
Q^{\sharp}(u_k,u_k)\to d_2,
$as $k\to+\infty$;
additionally, $b=d_1+d_2.$
By the definitions of $S^{\sharp}$ and $S^*$, we get
\begin{align*}
    d_1^{\frac{p}{p^*_s}} S^* \leqslant b = d_1+d_2, \qquad d_2^{\frac{1}{p^\sharp_s}}S^{\sharp}\leqslant b = d_1+d_2.
\end{align*}
From the first inequality we have 
$d_1^{\frac{p}{p^*_s}} S^* -d_1 \leqslant d_2 $, that is
\begin{align}
    \label{des:5.8achineses:sistemas}
     d_1^{\frac{p}{p^*_s}} \Big(S^* -d_1^{\frac{p^*_ -p}{p^*_s}}\Big) \leqslant d_2.
\end{align}
And from the second inequality we have 
$    d_2^{\frac{1}{p^\sharp_s}}S^{\sharp}-d_2\leqslant d_1 $, that is,
\begin{align}
    \label{des:5.8bchineses:sistemas}
    d_2^{\frac{1}{p^\sharp_s}}\Big(S^{\sharp}-d_2^{\frac{p^\sharp_s -1}{p^\sharp_s}}\Big) \leqslant d_1.
\end{align}

\begin{claim}
We have
\begin{align*}
    S^* -d_1^{\frac{p^*_s -p}{p^*_s}} >0, \qquad S^{\sharp} -d_2^{\frac{p^\sharp_s -1}{p^\sharp_s}} >0.
\end{align*}
\end{claim}
\begin{proof}
In fact, since $c+o(1)\|(u_k,v_k)\|=I(u_k,v_k)-\frac{1}{p}\langle I'(u_k,v_k),(u_k,v_k)\rangle$, we have
\begin{align*}
    I(u_k,v_k)-\frac{1}{p}\langle I'(u_k,v_k),(u_k,v_k)\rangle 
    & =  \left(\frac{1}{p}-\frac{1}{2p^{\sharp}_{\mu}(\alpha, \theta)}\right)Q^{\sharp}(u_k,u_k)+ \left(\frac{1}{p}-\frac{1}{p^*_s}\right) Q^*(u_k,v_k) \\
    &=c+o(1)\|(u_k,v_k)\|.
\end{align*}
Passing to the limit as $k\to+\infty$, we get
\begin{align}
    \label{eq:5.9chineses:sistemas}
    \left(\frac{1}{p}-\frac{1}{p^*_s}\right)  d_1 + \left(\frac{1}{p}-\frac{1}{2p^{\sharp}_{\mu}(\alpha, \theta)}\right) d_2 =c;
\end{align}
so,
\begin{align*}
    d_1 \leqslant 
    \left(\frac{1}{p}-\frac{1}{p^*_s}\right)^{-1}c =
    \frac{p(N-\beta)}{sp+\theta -\beta}\, c, \qquad d_2 \leqslant 
   \left(\frac{1}{p}-\frac{1}{2p^{\sharp}_{\mu}(\alpha, \theta)}\right)^{-1} c =\frac{2pp^\sharp_s}{2p^\sharp_s -p}\, c.
\end{align*}
Using these upper bounds for $d_1, d_2$ and the fact $0<c<c^*$, we have
\begin{align*}
    S^* - d_1 ^{\frac{p^*_s-p}{p^*_s}}  
    &\geqslant  S^*
    -\left[\frac{p(N-\beta)}{sp+\theta-\beta}\,c\right]^{\frac{p^*_s-p}{p^*_s}} 
    > S^*
    -\left[\frac{p(N-\beta)}{sp+\theta-\beta}\,c^*\right]^{\frac{p^*_s-p}{p^*_s}}\\
    &\geqslant  S^* 
    -\left[\frac{p(N-\beta)}{sp+\theta-\beta}\cdot \frac{(sp+\theta-\beta)}{p(N-\beta)}\,S^{*\frac{N-\beta}{(sp+\theta-\beta)}}\right]^{\frac{p^*_s-p}{p^*_s}}
     = S^*-S^* =0.
\end{align*}
Similarly,
\begin{align*}
     S^{\sharp} -d_2^{\frac{p^\sharp_s-1}{p^\sharp_s}}  &  \geqslant S^{\sharp} - \left[\frac{2pp^\sharp_s}{2p^\sharp_s-p}\,c\right]^{\frac{p^\sharp_s-1}{p^\sharp_s}}  
     > S^{\sharp} - \left[\frac{2pp^\sharp_s}{2p^\sharp_s-p}\,c^*\right]^{\frac{p^\sharp_s-1}{p^\sharp_s}}\\
    &  \geqslant S^{\sharp}\left[\frac{2pp^\sharp_s}{(2p^\sharp_s-p)}\cdot \frac{(2p^\sharp_s-p)}{2pp^\sharp_s}\,S^{\sharp \frac{p^\sharp_s}{p^\sharp_s-1}}\right]^{\frac{p^\sharp_s-1}{p^\sharp_s}}
     = S^{\sharp} - S^{\sharp} =0.
\end{align*}    
This concludes the proof of the claim.
\end{proof}

Following up,
inequalities~\eqref{des:5.8achineses:sistemas} and~\eqref{des:5.8bchineses:sistemas} imply, respectively, that
\begin{align*}
    \left[S^* 
    -\left(\frac{p(N-\beta)}{sp+\theta-\beta}c\right)^{\frac{p^*_s-p}{p^*_s}}\right] d_1^{\frac{p}{p^*_s}} \leqslant \left[S^* - d_1 ^{\frac{p^*_s-p}{p^*_s}}\right]d_1^{\frac{p}{p^*_s}} \leqslant d_2
\end{align*}
and
\begin{align*}
    \left[S^{\sharp} - \left(\frac{2pp^\sharp_s}{2p^\sharp_s-p}c\right)^{\frac{p^\sharp_s-1}{p^\sharp_s}}\right]d_2^{\frac{1}{p^\sharp_s}} \leqslant \left[S^{\sharp} -d_2^{\frac{p^\sharp_s-1}{p^\sharp_s}} \right]d_2^{\frac{1}{p^\sharp_s}} \leqslant d_1.
\end{align*}

If $d_1=0$ and $d_2=0$, then~\eqref{eq:5.9chineses:sistemas} implies that $c=0$, which is in contradiction with $c>0$. Therefore, $d_1>0$ and $d_2>0$ and we can choose $\epsilon _0>0$ such that $d_1\geqslant \epsilon _0 >0$ and $d_2\geqslant \epsilon _0 >0$; moreover, there exists a $K\in\mathbb{N}$ such that 
\begin{align*}
    Q^*(u_k,v_k) > \frac{\epsilon _0}{2}, \qquad Q^{\sharp} (u_k,u_k)>\frac{\epsilon _0}{2}
\end{align*}
for every $k > K$.
The embeddings~\eqref{imersoes}, and the improved Sobolev inequality~\eqref{des:1.10chineses:bis} imply that there exist $C_1,C_2>0$ such that 
\begin{align*}
    0<C_2\leqslant \|u_k\|_{L_M^{p,N-sp-\theta+pr}(\mathbb{R}^N,|y|^{-pr})}\leqslant C_1,
\end{align*}
where $r=\frac{\alpha}{p^*_s(\alpha, \theta)}$. 
For any $k>K$, we may find 
$\lambda _k>0$ and 
$x_k \in \mathbb{R}^N$ such that  
\begin{align*}
    \lambda_k^{(N-sp-\theta+pr)-N}\int_{B_{\lambda _k}(x_k)}\frac{|u_k(y)|^p}{|y|^{pr}} \dd y > \|u_k\|^p_{L_M^{p,N-sp-\theta+pr}(\mathbb{R}^N,|y|^{-pr})} -\frac{C}{2k} \geqslant \Tilde{C}>0.
\end{align*}
Now we define the sequence 
$\{\Tilde{u}_{k}\}_{k\in\mathbb{N}}\subset \Dot{W}^{s,p}_{\theta}(\mathbb{R}^N)$ by
$\Tilde{u}_k(x)=\lambda_k^{\frac{N-sp-\theta}{p}}u_k(\lambda _k x)$ and the sequence $\{\Tilde{v}_{k}\}_{k\in\mathbb{N}}\subset \Dot{W}^{s,p}_{\theta}(\mathbb{R}^N)$ by
$\Tilde{v}_k(x)=\lambda_k^{\frac{N-sp-\theta}{p}}v_k(\lambda _k x)$. As we have already shown,  $\|\Tilde{u}_k\|=\|u_k\|\leqslant C$ and $\|\Tilde{v}_k\|=\|v_k\|\leqslant C$ for every $k\in\mathbb{N}$; so, there exist  $u \in \Dot{W}^{s,p}_{\theta}(\mathbb{R}^N)$ and $v \in \Dot{W}^{s,p}_{\theta}(\mathbb{R}^N)$ such that, after passage to subsequence, still denoted in the same way, 
\begin{align*}
    \Tilde{u}_k \rightharpoonup u 
    & \text{ in $\Dot{W}^{s,p}_{\theta}(\mathbb{R}^N)$}
    & \quad \textup{and} \quad
    \Tilde{v}_k \rightharpoonup v 
    & \text{ in $\Dot{W}^{s,p}_{\theta}(\mathbb{R}^N)$}
\end{align*}
as $k\to+\infty$.
In a fashion similar to the proof of~\cite[Proposition $2-1$]{assunccao2020fractional}, we can prove that $u\not\equiv 0$ and $v\not\equiv 0$.

In addition, the boundedness of the sequences $\{\Tilde{u}_{k}\}_{k\in\mathbb{N}}\subset \Dot{W}^{s,p}_{\theta}(\mathbb{R}^N)$ and $\{\Tilde{v}_{k}\}_{k\in\mathbb{N}}\subset \Dot{W}^{s,p}_{\theta}(\mathbb{R}^N)$ implies that the sequences $\{|\Tilde{u}_k|^{p^*_s-2}\Tilde{u}_k\}_{k\in\mathbb{N}}\subset L^{\frac{p^*_s}{p^*_s-1}}(\mathbb{R}^N, |x|^{-\beta})$ and $\{|\Tilde{v}_k|^{p^*_s-2}\Tilde{v}_k\}_{k\in\mathbb{N}}\subset L^{\frac{p^*_s}{p^*_s -1}}(\mathbb{R}^N, |x|^{-\beta})$ are bounded too. In fact, by embeddings~\eqref{imersoes}, we obtain 
\begin{align*}
    \int_{\mathbb{R}^N} \frac{\left||\Tilde{u}_k|^{p^*_s-2}\cdot \Tilde{u}_k\right|^{\frac{p^*_s}{p^*_s -1}}}{|x|^{\beta}} \dd x   =\int_{\mathbb{R}^N} \frac{|\Tilde{u}_k|^{p^*_s}}{|x|^{\beta}} \dd x <C.
\end{align*}
and
\begin{align*}
    \int_{\mathbb{R}^N} \frac{\left||\Tilde{v}_k|^{p^*_s-2}\cdot \Tilde{v}_k\right|^{\frac{p^*_s}{p^*_s-1}}}{|x|^{\beta}} \dd x   =\int_{\mathbb{R}^N} \frac{|\Tilde{v}_k|^{p^*_s}}{|x|^{\beta}} \dd x <C.
\end{align*}

Then, after passage to subsequence, still denoted in the same way, we deduce that
\begin{align}
    \label{conv:5.10chineses:sistemas1}
    |\Tilde{u}_k|^{p^*_s-2}\Tilde{u}_k\ \rightharpoonup |u|^{p^*_s-2}u\ \quad \textup{in $L^{\frac{p^*_s}{p^*_s -1}}(\mathbb{R}^N, |x|^{-\beta})$}
\end{align}
and
\begin{align}
    \label{conv:5.10chineses:sistemas}
    |\Tilde{v}_k|^{p^*_s-2}\Tilde{v}_k\ \rightharpoonup |v|^{p^*_s-2}v\ \quad \textup{in $L^{\frac{p^*_s}{p^*_s -1}}(\mathbb{R}^N, |x|^{-\beta})$}
\end{align}
as $k\to+\infty$.

For any $\phi_1, \phi_2 \in L^{p^*_s(\alpha, \theta)}(\mathbb{R}^N, |x|^{-\alpha})$, Lemma $2.5$ in~\cite{assunccao2023fractional} implies that
\begin{align}
    \label{lim:5.11chineses:sistemas1}
    \begin{split}
    \lim\limits_{k \to \infty} \int_{\mathbb{R}^N} \left[I_{\mu}\ast F_{\alpha}(\cdot , \Tilde{u}_k)\right](x) f_{\alpha}(x,\Tilde{u}_k)\phi_1 (x) \dd x \\
    = \int _{\mathbb{R}^N} \left[I_{\mu}\ast F_{\alpha}(\cdot , u)\right](x) f_{\alpha}(x,u)\phi_1 (x) \dd x
    \end{split}
\end{align}
and
\begin{align}
    \label{lim:5.11chineses:sistemas}
    \begin{split}
    \lim\limits_{k \to \infty} \int_{\mathbb{R}^N} \left[I_{\mu}\ast F_{\alpha}(\cdot , \Tilde{v}_k)\right](x) f_{\alpha}(x,\Tilde{v}_k)\phi_2 (x) \dd x \\
    = \int _{\mathbb{R}^N} \left[I_{\mu}\ast F_{\alpha}(\cdot , v)\right](x) f_{\alpha}(x,v)\phi_2 (x) \dd x. 
    \end{split}
\end{align}
Since $\Dot{W}^{s,p}_{\theta}(\mathbb{R}^N) \hookrightarrow L^{p^*_s(\alpha, \theta)}(\mathbb{R}^N,|x|^{-\alpha})$, \eqref{lim:5.11chineses:sistemas1} and \eqref{lim:5.11chineses:sistemas} hold for any $\phi_1,\phi_2 \in \Dot{W}^{s,p}_{\theta}(\mathbb{R}^N)$.

Finally, we need to check that $\{\Tilde{u}_k\}_{k \in \mathbb{N}}\subset\Dot{W}^{s,p}_{\theta}(\mathbb{R}^N)$ and $\{\Tilde{v}_k\}_{k \in \mathbb{N}}\subset\Dot{W}^{s,p}_{\theta}(\mathbb{R}^N)$ are also a $(PS)_{c}$ sequence for the functional $I$ at energy level $c$. Do to this, we note that the norms in $L^{p^*_s(\alpha, \theta)}(\mathbb{R}^N,|x|^{-\alpha})$ are invariant under the special dilatation $\Tilde{u}_k=\lambda _k^{\frac{N-sp-\theta}{p}}u_k(\lambda _k x)$ and $\Tilde{v}_k=\lambda _k^{\frac{N-sp-\theta}{p}}v_k(\lambda _k x)$. In fact
\begin{align*}
    \|\Tilde{u}_k\|^{p^*_s(\alpha, \theta)}_{L^{p^*_s(\alpha, \theta)}} 
 = \int_{\mathbb{R}^N} \frac{\lambda_k^{\frac{N-sp-\theta}{p}p^*_s(\alpha, \theta)}|u_k(\lambda_k x )|^{p^*_s(\alpha, \theta)})}{|x|^{\alpha}} \dd x = \int_{\mathbb{R}^N} \frac {|u_k(\Bar{x})|^{p^*_s(\alpha, \theta)}}{|\Bar{x}|^{\alpha}} \dd \Bar{x} =  \|u_k\|^{p^*_s(\alpha, \theta)}_{L^{p^*_s(\alpha, \theta)}}
\end{align*}
and
\begin{align*}
    \|\Tilde{v}_k\|^{p^*_s(\beta, \theta)}_{L^{p^*_s(\beta, \theta)}} 
 = \int_{\mathbb{R}^N} \frac{\lambda_k^{\frac{N-sp-\theta}{p}p^*_s(\beta, \theta)}|v_k(\lambda_k x )|^{p^*_s(\beta, \theta)})}{|x|^{\beta}} \dd x = \int_{\mathbb{R}^N} \frac {|v_k(\Bar{x})|^{p^*_s(\beta, \theta)}}{|\Bar{x}|^{\beta}} \dd \Bar{x} =  \|v_k\|^{p^*_s(\beta, \theta)}_{L^{p^*_s(\beta, \theta)}}.
\end{align*}
Additionally, we have
\begin{align*}
    \int_{\mathbb{R}^N} \frac{\eta 
    |\Tilde{u}_k(x)| ^a|\Tilde{v}_k(x)| ^b}{|x|^{\beta}} \dd x 
    & = \int_{\mathbb{R}^N} \frac{\eta \lambda _k ^{\frac{N-sp-\theta}{p}(a+b)}|{u}_k(\lambda_k x)|^a|v_k(\lambda_k x)|^b}{|x|^\beta} \dd x \\
    &= \int_{\mathbb{R}^N}\frac{\eta |u_k (\Bar{x})|^a|v_k (\Bar{x})|^b}{|\Bar{x}|^{\beta}} \dd \Bar{x}.
\end{align*}
Thus, we have
\begin{align*}
    \lim\limits_{k \to + \infty} I(\Tilde{u}_k, \Tilde{v}_k) =c.
\end{align*}
Moreover, for all $ \phi_1, \phi_2 \in \Dot{W}^{s,p}_{\theta}(\mathbb{R}^N)$, we have $\phi_{1,k} (x) = \lambda_k^{\frac{N-sp-\theta}{p}}\phi_1 \left(x/\lambda _k\right) \in \Dot{W}^{s,p}_{\theta}(\mathbb{R}^N)$ and $\phi_{2,k} (x) = \lambda_k^{\frac{N-sp-\theta}{p}}\phi_2 \left(x/\lambda _k\right) \in \Dot{W}^{s,p}_{\theta}(\mathbb{R}^N)$. From $I'(u_k, v_k) \to 0$ in $W'$ as $k\to+\infty$, we can deduce that 
\begin{align*}
    \lim\limits_{k \to + \infty} \langle I'(\Tilde{u}_k, \Tilde{v}_k), (\phi_1, \phi_2)\rangle = \lim \limits_{k \to + \infty} \langle I'(u_k, v_k), (\phi_1, \phi_2)\rangle =0.
\end{align*}

Thus~\eqref{conv:5.10chineses:sistemas1},~\eqref{conv:5.10chineses:sistemas},~\eqref{lim:5.11chineses:sistemas1} and~\eqref{lim:5.11chineses:sistemas} lead to
\begin{align*}
    \langle I'(u,v),(\phi_1, \phi _2) \rangle = \lim \limits_{k \to + \infty} \langle I'(\Tilde{u}_k, \Tilde{v}_k), (\phi_1, \phi _2) \rangle =0.
\end{align*}
Hence $(u,v)$ is a nontrivial weak solution to problem~\ref{prob:sistemas}.
\end{proof}

\begin{proof}[Proof of Theorem~\ref{teo:sistemas02}]
    The proof follows the same steps of the proof of Theorem~\ref{teo:sistemas}. Here we only remark that for problem~\eqref{prob:sistemas02} with a Hardy potential and double Sobolev type nonlinearities we have to define the value below which we can recover the compactness of the Palais-Smale sequences by
    \begin{align*}
c^* \coloneqq \min_{k\in\{1,2\}}  \biggl\{
\Bigl(\frac{1}{p}-\frac{1}{p^{\ast}_{s}}\Bigr)
    S^{*}(N, s, \gamma, \beta_k)^{\frac{p^{\ast}_{s}(\beta_k,\theta)}{p^{\ast}_{s}(\beta_k,\theta)-p}}\biggr\}.
\end{align*}
Similarly, for problem~\eqref{prob:sistemas03} with a Hardy potential and double Choquard type nonlinearities we have to define the corresponding number by
\begin{align*}
c^* \coloneqq \min_{k\in\{1,2\}}  \biggl\{
\Bigl(\frac{1}{p}-\frac{1}{2p^{\sharp}_{s}(\delta, \theta, \mu_k)}\Bigr)
    S^{\sharp}(N, s, \gamma, \beta_k)^{\frac{2p^{\sharp}_{s}(\delta, \theta, \mu_k)}{2p^{\sharp}_{s}(\delta, \theta, \mu_k)-p}}\biggr\}.
\end{align*}
The details are omitted.
\end{proof}

\bibliographystyle{IEEEtranS}
\bibliography{bibliografia}

\end{document}